\documentclass[reqno, english,11pt]{amsart}
\RequirePackage{mathrsfs} \let\mathcal\mathscr
\usepackage{a4wide}
\usepackage[all]{xy}
\usepackage{enumerate}
\usepackage{color}
\usepackage{amssymb,amsmath,amsfonts,amsthm}

\usepackage{hyperref}
\usepackage{dsfont}
\usepackage{upgreek}
\usepackage{mathrsfs}
\usepackage{mathabx,yfonts}
\usepackage{enumerate}
\usepackage{enumitem}
\usepackage{graphicx}

\DeclareMathOperator{\Pic}{Pic} 
\DeclareMathOperator{\Br}{Br} 
\DeclareMathOperator{\rank}{rank} 
\DeclareMathOperator{\chr}{char} 
\DeclareMathOperator{\adj}{adj} 
 
\DeclareMathOperator{\Gal}{Gal} 

\newtheorem{theorem}{Theorem}

\newtheorem{lemma}[theorem]{Lemma}
\newtheorem{proposition}[theorem]{Proposition}
\newtheorem{corollary}[theorem]{Corollary}
\theoremstyle{definition}
\newtheorem{definition}[theorem]{Definition}
\newtheorem{remark}[theorem]{Remark}

\newtheorem*{notation}{Notation}

\numberwithin{theorem}{section}
\numberwithin{equation}{section}
\numberwithin{table}{section}

\DeclareSymbolFont{bbold}{U}{bbold}{m}{n}
\DeclareSymbolFontAlphabet{\mathbbold}{bbold}

\renewcommand{\P}{\mathbb{P}}

\newcommand{\Q}{\mathbb{Q}}
\newcommand{\F}{\mathbb{F}}

\newcommand{\N}{\mathbb{N}}

\newcommand{\R}{\mathbb{R}}

\renewcommand{\l}{\left}

\renewcommand{\r}{\right}

\renewcommand{\c}{\mathcal}
\renewcommand{\epsilon}{\varepsilon}

\renewcommand{\leq}{\leqslant}
\renewcommand{\geq}{\geqslant}
\renewcommand{\#}{\sharp}
\renewcommand{\gg}{\ggg}

\renewcommand{\ll}{\lll}

\newcommand\vz{\mathbf{0}}
\newcommand\E{\mathbf{E}}

\newcommand\FF{\mathbb{F}}
\newcommand\PP{\mathbb{P}}
\newcommand\ZZ{\mathbb{Z}}

\newcommand\QQ{\mathbb{Q}}
\newcommand\RR{\mathbb{R}}
\newcommand\CC{\mathbb{C}}
\newcommand\GG{\mathbb{G}}
\newcommand\Ga{\GG_\mathrm{a}}
\newcommand\Gm{\GG_\mathrm{m}}
\newcommand{\Kbar}{\overline{K}}
\newcommand{\OO}{\mathcal{O}}
\newcommand{\places}{\Omega_K}
\newcommand{\archplaces}{{\Omega_\infty}}

\newcommand{\vv}{\mathbf{v}}
\newcommand{\vq}{\mathbf{q}}
\newcommand{\vx}{\mathbf{x}}
\newcommand{\vy}{\mathbf{y}}

\newcommand{\ideals}{\mathcal{I}_K}
\newcommand{\id}[1]{\mathfrak{#1}}
\newcommand{\ppp}{\id{p}}
\newcommand{\PPP}{\id{P}}
\newcommand{\aaa}{\id{a}}
\newcommand{\bbb}{\id{b}}
\newcommand{\ccc}{\id{c}}
\newcommand{\ddd}{\id{d}}

\newcommand{\norm}{\mathfrak{N}}
\newcommand{\where}{\ :\ }
\newcommand{\classrep}{\mathfrak{r}}
\newcommand{\abs}[1]{\left|#1\right|}
\newcommand{\absv}[1]{\left|#1\right|_v}
\newcommand{\vecnorm}[1]{\left\lVert #1 \right\rVert}
\newcommand{\vecnormv}[1]{\left\lVert #1 \right\rVert_v}

\newcommand{\one}{{\mathbf{1}}}

\newcommand\qr[2]{\left(\frac{#1}{#2}\right)}
\newcommand\qrp[1]{\qr{#1}{\ppp}}
\newcommand\card{\#}
\newcommand\DD{\mathcal{D}}
\newcommand\www{\id{W}}

\newcommand{\locdegv}{{m_v}}
\newcommand\dg{m}
\newcommand\qfheight[1]{{\langle #1 \rangle}}
\newcommand{\bomega}{\boldsymbol{\omega}}
\newcommand{\vxi}{\boldsymbol{\xi}}
\newcommand{\nonarchplaces}{{\Omega_0}}

\newcommand{\locheightv}[1]{H_v(#1)}
\newcommand{\freeunits}{\mathcal{A}}

\DeclareMathOperator{\res}{Res}

\DeclareMathOperator{\vol}{vol}
\DeclareMathOperator\supp{supp}

\newcommand\WHERE{\,\Bigg|\,}

\begin{document}

\title[
Rational points of bounded height on general conic bundle surfaces
]
{
Rational points of bounded height on general conic bundle surfaces
}

\author{Christopher Frei}
\address{
University of Manchester \\
School of Mathematics\\
Oxford Road\\
Manchester\\
M13 9PL\\
UK
}
\email{christopher.frei@manchester.ac.uk}

\author{Daniel Loughran}
\address{
University of Manchester \\
School of Mathematics\\
Oxford Road\\
Manchester\\
M13 9PL\\
UK}
\email{daniel.loughran@manchester.ac.uk}

\author{Efthymios Sofos}
\address{
Max-Planck-Institut f\"{u}r Mathematik\\
Vivatsgasse 7\\
Bonn\\
53111\\
Germany
}
\email{e.sofos@math.leidenuniv.nl}

\begin{abstract}
A conjecture of Manin predicts the asymptotic distribution 
of rational points of bounded height on Fano varieties.  
In this paper we use conic bundles to obtain
correct lower bounds for a wide class of surfaces over number fields
for which the conjecture is still far from being proved.
For example, we obtain the conjectured lower bound of Manin's conjecture for any del Pezzo surface
whose Picard rank is sufficiently large, or for arbitrary del Pezzo surfaces
after possibly an extension of the ground field of small degree.
\end{abstract}


\maketitle

\setcounter{tocdepth}{1}
\tableofcontents

\section{Introduction}\label{intro}

\subsection{Manin's conjecture}
Manin's conjecture, first posed in \cite{fmt} and developed further in \cite{BM90} and \cite{peyre}, predicts
precise asymptotic behaviour for the number of rational points of bounded height on Fano varieties
and similar varieties. Recall that a \emph{Fano variety} over a number field $K$ is a smooth projective variety $X$ over $K$
with ample anticanonical divisor $-K_X$. The theory of height functions gives rise
to a choice of anticanonical height $H$ on $X$, which has the property that the cardinality
\[
N_{U,H}(B) := \card\{x \in U(K) \where H(x) \leq B\}
\]
is finite for all open subsets $U \subset X$ and all $B>0$. If $X(K) \neq \emptyset$, then 
Manin's conjecture predicts the existence of an open subset $U \subset X$ and a positive constant $C_{X,H}$ such that
\begin{equation} \label{conj:Manin}
N_{U,H}(B) \sim C_{X,H}B(\log B)^{\rho(X)-1}, \quad \mbox{as } B \to \infty,
\end{equation}
where $\rho(X)$ is the rank of the Picard group of $X$. One needs to restrict
to an open subset in general to avoid so-called \emph{accumulating subvarieties} that
contain more than the expected number of rational points (e.g.~lines on cubic surfaces).

Whilst there are counterexamples to Manin's conjecture \cite{BT96}, 
it is nonetheless expected to hold in dimension $2$.

\subsection{Del Pezzo surfaces}
\label{sec:del pezzo}
A Fano variety $X$ of dimension $2$ is called a \emph{del Pezzo surface}. 
An important invariant of del Pezzo surfaces is their \emph{degree} $d = (K_X)^2$.
This satisfies $1 \leq d \leq 9$, with surfaces of smaller degree generally having a more complicated
arithmetic and geometry (for example $1 \leq \rho(X) \leq 10 - d$).
The expectation is that the asymptotic formula \eqref{conj:Manin} should hold with $U$ taken to be the
complement of the lines when $d \neq 1$, and $U$ taken to be the complement of the lines
and the singular elements of the linear system $|-K_X|$ when $d=1$ (by a \emph{line}, we mean a $(-1)$-curve on  $X$).

Our emphasis is on obtaining lower bounds of the correct order of magnitude. Our first result achieves this
if
the rank of the Picard group is large enough. 

\begin{theorem}\label{thm:rho}
	Let $1 \leq d \leq 5$ and let $\rho_d$ be given by the following table.
	\begin{table}[ht]
	\centering
	$\begin{array}{|l|lllll|}
		\hline
		d & 5 & 4 & 3 & 2 & 1 \\	\hline
		\rho_d & 3 & 4 & 4 & 5 & 6 \\ \hline	
	\end{array}$
	\caption{}
	\label{tab:rho_d}
	\vspace{-0.6cm}
	\end{table} 
	
	Let $X$ be a del Pezzo surface of degree $d$ over a number field $K$
	with $\rho(X) \geq \rho_d$. If $d$ is even, assume that $X(K) \neq \emptyset$.
	Then for all anticanonical height functions $H$ on $X$ we have
		\[
N_{U,H}(B) 
\gg _H
B 
(\log B)^{\rho(X)-1},
\
\text{as}
\
B \to \infty
.\]
\end{theorem} 

No sharp upper or lower bounds were previously known for \emph{any} del Pezzo surface of degree $1$ or $2$. In contrast, Theorem \ref{thm:rho} obtains for the first time sharp lower bounds for Manin's conjecture for \emph{all split} del Pezzo surfaces of arbitrary degree, i.e.~ those with all lines defined over the ground field (such surfaces are exactly those with a rational point and  $\rho(X) = 10-d$).

Even in the very special case $d=3$ and $K=\QQ$, Theorem \ref{thm:rho} is not covered by  the best result known so far, due to Slater and Swinnerton-Dyer~\cite{slat}. They proved the correct lower bound under the hypothesis that $X$ contains $2$ skew lines defined over $\QQ$, using the fact that $X$ is birational to $\P^1 \times \P^1$ to provide a suitable parametrisation of all rational points. The result was subsequently extended to all number fields in~\cite{ffrei} using a conic bundle structure with a section on $X$. The Fermat cubic surface $x_0^3+x_1^3+x_2^3+x_3^3=0$ over $\QQ$ is, for example, covered by Theorem \ref{thm:rho} but not by \cite{slat,ffrei}.

Finally, no sharp lower bounds were known for any del Pezzo surface of degree $d\geq 4$ previously to our work, apart from those cases where Manin's conjecture is known. Any del Pezzo surface of degree $d\geq 6$ is toric,
hence Manin's conjecture
here follows from the general result \cite{toric}. Manin's conjecture is known for split del 
Pezzo surfaces of degree $5$ \cite{regis5} over $\QQ$ and a single del Pezzo surface of degree $4$ \cite{BB11} over $\QQ$; it is not
known for any del Pezzo surface of degree $d \leq 3$ over $\QQ$, and for no surface of degree $d \leq 5$
over number fields other than $\QQ$. There are however results for some \emph{singular}
del Pezzo surfaces; we do not consider singular surfaces in this paper, as we prefer to focus on the more
difficult case of smooth surfaces.

The known upper bounds are still very far from the conjectured truth for $d\leq 3$.
For cubic surfaces, the strongest upper bounds are due to Heath-Brown \cite{rogerbest} who showed that
$N_{U,H}(B)\ll_{\varepsilon,X}B^{4/3+\varepsilon}$ for any $\varepsilon>0$,
provided that $X$ is defined over $\Q$ and contains $3$ coplanar lines.
For $d=2$, Salberger~\cite{salbergerunpub} has proved that 
$N_{U,H}(B)\ll_{\varepsilon,X}B^{3/\sqrt{2}+\varepsilon}$ for any $\varepsilon>0$.
Browning and Swarbrick-Jones~\cite[Thm.~1.3]{mikey}
have shown that 
$N_{U,H}(B)\ll_{\varepsilon,X}B^{2+\varepsilon}$ for any $\varepsilon>0$,
whenever $X$ is equipped with a conic bundle structure over $K$.
In the case that $X$ is \emph{split}, Salberger has proved that
$N_{U,H}(B)\ll_{\varepsilon,X}B^{11/6+\varepsilon}$ for any $\varepsilon>0$,
as announced in
the conference ``G\'{e}om\'{e}trie arithm\'{e}tique et vari\'{e}t\'{e}s rationnelles'' at Luminy in 2007.
For $d=1$, Mendes da Costa \cite[Prop.~4]{MdC13} has shown
that $N_{U,H}(B)\ll B^{3 - \delta}$, for some $\delta > 0$.



For $d=4$, much better bounds are known. As already mentioned,
Manin's conjecture has
been established for a  quartic del Pezzo surface \cite{BB11}.
During the conference ``Higher dimensional varieties and rational points''
at Budapest in 2001, Salberger has announced 
a proof of 
$N_{U,H}(B)\ll_{\varepsilon,X}B^{1+\varepsilon}$ for any $\varepsilon>0$, 
whenever $X$ contains a conic over $\Q$, and this work was subsequently extended to number fields
by Browning and Swarbrick-Jones~\cite[Thm.~1.1]{mikey}.
Even more is known now for $d=4$ due to very recent work of Browning and Sofos \cite{BS16}, building in part upon the results of the present paper.

We are also able obtain to obtain the conjectured lower bound \textit{for all} del Pezzo surfaces 
after a finite extension of the base field $K$;
again, 
no previous result on this topic existed for del Pezzo surfaces of $d=1$ or $2$ (see \cite{ffrei} for $d=3$).

\begin{theorem} \label{thm:degree}
	Let $1 \leq d \leq 5$ and let $n_d$ be given by the following table.
	\begin{table}[ht]
	\centering
	$\begin{array}{|l|lllll|}
		\hline
		d & 5 & 4 & 3 & 2 & 1 \\	\hline
		n_d & 5 & 80 & 432 & 4032 & 138240 \\ \hline	
	\end{array}$
	\caption{}
	\label{tab:n_d}
	\vspace{-0.6cm}
	\end{table} 
	
	Let $X$ be a del Pezzo surface of degree $d$ over a number field $K$.
	Then there exists a finite extension $L_0/K$ of degree at most $n_d$ such that
	for all number fields $L_0 \subset L$ and all anticanonical height functions $H$ on $X_L$, we have
	\[
N_{U_L,H}(B) 
\gg_{H} 
B 
(\log B)^{\rho(X_L)-1},
\
\text{as}
\
B \to \infty
.\]
\end{theorem}

\subsection{Counterexamples to Manin's conjecture} \label{sec:BT}
Batyrev-Tschinkel \cite{BT96} used lower bounds for the Fermat cubic
surface to obtain counterexamples to Manin's conjecture over any number field which
contains a third root of unity. Theorem \ref{thm:rho} yields for the first time the 
correct lower bound for Manin's conjecture for the Fermat cubic surface over any number 
field, and in particular over $\QQ$.
As an application, we are able to extend their counterexample to arbitrary
number fields, and moreover to improve upon the lower bounds  obtained in \cite[Thm.~3.1]{BT96}
(other counterexamples  over arbitrary number fields 
have also been constructed in \cite{LR14} and \cite{Lou15}).

\begin{theorem}	\label{thm:BT}
	Let $K$ be a number field and 
	$$Y: \quad a_0x_0^3 + a_1x_1^3 + a_2x_2^3 + a_3x_3^3 = 0 \quad \subset \PP^3_K \times \PP^3_K.$$
	For any dense open subset $U \subset Y$ and any anticanonical height function $H$ on $Y$ we have
	$$N_{U,H}(B) \gg_{U,H} 
	\begin{cases}
		B (\log B)^6, &\quad \text{if } \QQ(\sqrt{-3}) \subseteq K, \\
		B (\log B)^3, &\quad \text{if } \QQ(\sqrt{-3}) \not\subset K.
	\end{cases}$$
\end{theorem}

As explained in \cite[\S 1]{BT96}, this yields a counterexample to Manin's conjecture as $Y$ is a smooth Fano variety with $\rho(Y) = 2$.

\subsection{Conic bundle surfaces}\label{sec:conic bundles}
Our results on del Pezzo surfaces are proved via a
fibration method using \emph{conic bundles}.
Conic bundles have been used to great success
in special cases
of Manin's conjecture \cite{BB11, BBP12}
and for obtaining upper bounds \cite{rogerbest}; our aim is to 
construct techniques that enable us 
to deduce 
lower bounds in much higher generality. 

In this paper, we use the following definition of conic bundle surfaces.

\begin{definition} \label{def:conic_bundle}
	A conic bundle over a field $k$ is a smooth projective surface $X$ over $k$
	together with a dominant morphism $$\pi: X \to \PP^1,$$
	all of whose fibres are isomorphic to plane conics. 
	We define the \emph{complexity} of $\pi$ to be
	\begin{equation*}
	 c(\pi):= \hspace{-15pt} \sum_{\substack{P \ \in  \ \P^1 \\ X_P  \ \text{is non-split}}} \hspace{-15pt}[k(P) : k],
	\end{equation*}
	where $k(P)$ denotes the residue field of $P$. 
\end{definition}	

Recall that a conic $C$ over a field $k$ is called \emph{split} if it is either smooth
or isomorphic to two rational lines over $k$. 
The complexity $c(\pi)$ is a rough measure of the arithmetic
difficulty of the conic bundle. For example if $c(\pi) = 0$ and $\pi$ admits a smooth
fibre with a rational point, then $\pi$ has a section (this can be proved using Grothendieck's purity theorem 
and the fact that  $\Br \P^1_k = \Br k$).
Note that what we have called the complexity of $\pi$ is often
referred to as the \emph{rank} of $\pi$ in the literature
(see \cite[Thm.~0.4]{skorr}). 
We prefer the former
terminology to avoid any possible confusion with the Picard rank of $X$.

Our main theorem concerns counting rational points of bounded height on conic bundle surfaces.
In this level of generality
the difficulties in obtaining 
precise results towards Manin's conjecture
are formidable.
Serre~\cite{serre} has shown that, 
unless we are in the very special 
case
that
the conic bundle has a section, 
only $0\%$ of the conics in the family
have a rational point. 
All the cubic surfaces in \cite{slat} and \cite{ffrei}
have a conic bundle with a section, which was instrumental in proving the lower bound.
Similarly in the proof of Manin's conjecture for a quartic del Pezzo surface \cite{BB11}, there 
was a conic bundle with a section.
The main
difficulty in our paper is about giving precise bounds when the conic bundle
\emph{does not have a section}, so that very few fibres have a rational point.
We work over arbitrary number fields for completeness, but this fundamental difficulty
regarding the non-existence of a section is already present over $\QQ$.
A further level of difficultly arises in working with 
\emph{arbitrary conic bundles surfaces},
rather than just special classes of del Pezzo surfaces with a conic bundle structure (e.g.~cubic surfaces with a line).

We are able to deal with arbitrary conic bundles surfaces by working out explicit
equations inside $\PP^2$-bundles over $\PP^1$ and explicitly calculating the relevant height
functions (this is achieved in \S \ref{sec:conic_bundles_geometry} and \S \ref{sec:del_Pezzo}).
To overcome the problem regarding the non-existence of a section we
construct certain arithmetic functions 
that are able
to detect
asymptotically the correct proportion of 
conics with a rational point for any conic bundle surface.
These detector functions translate the problem of counting points into one of 
estimating suitable divisor sums, which is a completely independent and very classical topic 
in analytic number theory.
Our approach 
yields several new results that were out of reach of previous methods, even over $\QQ$.
It has already found further applications to Manin's conjecture for quartic del Pezzo surfaces in the recent paper \cite{BS16}.

The theory of the resulting divisor sums is treated in the companion paper \cite{divsumpub},
written by the first and last named author.
In \emph{ibid}.~the complexity of a system of binary forms is defined, and a general lower bound
conjecture for divisor sums is stated.
The present paper's central result obtains sharp
lower bounds for Manin's conjecture for conic bundle surfaces, conditional on this 
conjecture.

\begin{theorem}\label{thm:conic bundle}
 Let $c \in \N$ and assume the 
validity of \cite[Conjecture 1]{divsumpub}
for systems of binary forms of complexity $c$.
Let $\pi:X\to \P^1_K$ be a conic bundle of complexity $c$ with a rational point lying on a smooth fibre. Then for any choice of anticanonical height function $H$ on $X$ and any non-empty open subset $U \subset X$, we have
  \begin{equation*}
    N_{U,H}(B)\gg_{U,H} B(\log B)^{\rho(X) - 1}, \quad \text{as } B \to \infty,
  \end{equation*}
  where $\rho(X)$ is the rank of the Picard group of $X$.
\end{theorem}
Note that such surfaces need no longer be Fano. However, the results here are compatible with the more general
framework for Manin's conjecture presented in \cite{BM90} (namely \cite[Conj.~C']{BM90}).
Theorem \ref{thm:conic bundle} is non-trivial only if $-K_X$ is big (this property holds for example if $K_X^2 > 0$ \cite[Rem.~2.10]{TVAV11}).
It is in such cases that there
are only finitely many rational points of bounded height on some non-empty open subset $U$, and the lower bounds
obtained in Theorem \ref{thm:conic bundle} are conjecturally sharp provided $U$ is taken sufficiently small.
Note that Theorem \ref{thm:conic bundle} applies to \emph{any}
non-empty
open subset $U$. 
In particular, this gives a conditional proof of Zariski density of rational points
on conic bundle surfaces
over number fields,
assuming the existence of a rational point on
a
smooth fibre; this property is currently open
in most cases (see \cite{KM16} and \cite{lilian_annals} for recent results).

As was proved in \cite[Thm.~1.1]{divsumpub}, the lower bound conjecture holds for systems of forms of complexity at most $3$.
From Theorem \ref{thm:conic bundle}, this allows us to obtain the following unconditional result.

\begin{theorem}\label{thm:conic bundle small}
  Let $\pi:X\to \P^1_K$ be a conic bundle of complexity at most $3$
  with a rational point.  
  Then for any choice of anticanonical height function $H$ on $X$ and any non-empty open subset $U \subset X$, we have
  \begin{equation*}
    N_{U,H}(B)\gg_{U,H} B(\log B)^{\rho(X) - 1}.
  \end{equation*}
\end{theorem}

For comparison, it is known that any conic bundle $X$ of complexity at most $3$ with a rational point is rational (see \cite[\S1]{KM16}
and the references therein). In particular the set of rational points, when non-empty, is
Zariski dense. Theorem \ref{thm:conic bundle small} should therefore be compared with the ``trivial''
lower bound $N_{U,H}(B)\gg_{U,H} B$, given by considering the rational points on some smooth conic in $X$ which meets $U$.
Note that the trivial lower
bound is never sharp, as for a conic bundle $X$  one always has $\rho(X) \geq 2$ (see Lemma \ref{lem:Pic}).

One nice feature of Theorem \ref{thm:conic bundle small} is that we are able to obtain lower bounds for Manin's conjecture for some surfaces of 
\emph{arbitrarily large Picard number}. 
For example, take $X$ to be the blow-up of $\PP^2_K$ in $n+1$ rational points, of which $n$ lie on a conic and no $3$ are collinear. This is equipped with a conic bundle arising from the linear system of lines in $\PP^2_K$ through one of the blown-up points. This conic bundle has complexity $0$ and  $X$ has Picard rank $n+2$. Moreover \cite[Rem.~2.10]{TVAV11} implies that $-K_X$ is big. Using the methods of \cite{DL10,DL15}, one can also  show that for sufficiently large $n$, such $X$ are not equivariant compactifications of $\Ga^2, \Gm^2$ nor $\Ga \rtimes \Gm$. In particular, Manin's conjecture is not already known for such $X$  as specials cases of \cite{toric}, \cite{CLT02}, \cite{TT12}. These appear to be the first examples of surfaces of arbitrarily large Picard number which are not equivariant compactifications for which lower bounds for Manin's conjecture of the correct order of magnitude have been obtained.

The results stated in \S \ref{sec:del pezzo} will be proved using Theorem \ref{thm:conic bundle small}, 
together with a geometric analysis of the configurations of lines on del Pezzo surfaces.

\begin{remark} \label{rem:Galois_actions}
	Let $1 \leq d \leq 5$ and let $X$ be a del Pezzo surface of degree $d$ over $K$.
	We recall the definition of the graph of lines of $X$ \cite[\S26.9]{Man86}.
	This has one vertex for each line on $X_{\bar{K}}$,
	with $e$ edges between two distinct vertices corresponding to lines $L_1$ and $L_2$ with $L_1 \cdot L_2 = e$.
	This graph only depends on $d$, up to isomorphism, and we denote it by $G_d$ (this can be canonically constructed
	using the theory of root systems). The graph $G_d$ has automorphism group the Weyl group $W(\E_{9-d})$ 
	\cite[Thm.~23.8]{Man86}  (we follow
	Dolgachev's convention \cite[\S8.2.2]{Dol12} and define the root systems $\E_r$ for any $3 \leq r \leq 8$). 	
	This graph can be constructed in \texttt{Magma} using the method described in \cite[\S3.1.1]	{JL15}.
	
	The absolute Galois group $\Gal(\bar{K}/K)$ naturally acts
	on the lines of $X$ over $\bar{K}$, hence the Galois action gives rise to a subgroup of $W(\E_{9-d})$,
	which is well-defined up to conjugacy.
	We will show in Proposition \ref{prop:dp_conic_bundle} that 
	this action	determines whether or not the surface admits a conic bundle, at least under the additional
	assumption that $X(K) \neq \emptyset$.
	We enumerated all such conjugacy classes in \texttt{Magma} 
	(the relevant code can be found on the second-named
	author's web page). A summary of the results can be found in Table \ref{tab:E}.
	
	\begin{table}[ht]
	\centering
	$\begin{array}{|r|r|r|r|r|}
		\hline
		d & \text{Subgroups} & \text{Conic bundle } & c(\pi) = 0 &
		c(\pi) \leq 3 \\ \hline
		5 & 19 & 11 & 4 & 11 \\
		4 & 197 & 73 & 18 & 23 \\    
		3 & 350 & 172 & 19 & 41 \\  
		2 & 8074 & 1791 & 87 & 165 \\
		1 & 62092 & 6356 & 96 & 221 \\ \hline    
	\end{array}$
	\caption{}
	\label{tab:E}
	\vspace{-0.6cm}
	\end{table} 
	
	In Table \ref{tab:E}, the first column gives the total number of \emph{conjugacy classes} of subgroups of $W(\E_{9-d})$.
	The second column the number which admit a conic bundle. The third column the number which
	admit a conic bundle $\pi$ of complexity $0$, and the last column the number which
	admit a conic bundle $\pi$ of complexity at most $3$ (this latter column is the number
	of conjugacy classes of subgroups to which Theorem \ref{thm:conic bundle small} applies).
\end{remark}

\subsection{Plan of the paper}
In~\S\ref{sec:conic_bundles_geometry} we recall and prove some necessary facts
about the geometry of conic bundles. In particular, how to write down equations
for conic bundles inside $\PP^2$-bundles over $\PP^1$.
In~\S\ref{sec:conics}, we prove some results about counting rational points on conics, 
which will be needed in the proof of Theorem \ref{thm:conic bundle}. 

In~\S\ref{sec:conic bundle counting}, we will prove our main result (Theorem \ref{thm:conic_bundle_precise}),
which is a more
precise
version of Theorem \ref{thm:conic bundle}.

We finish with \S\ref{sec:del_Pezzo}, which concerns del Pezzo surfaces with a conic bundle structure. 
We apply Theorem
\ref{thm:conic bundle small}
to prove the results stated in \S\ref{sec:del pezzo} and \S\ref{sec:BT}. 
In \S\ref{sec:del_Pezzo}, we also give a complete classification
of the equations for del Pezzo surfaces with a conic bundle, inside an appropriate $\PP^2$-bundle
over $\PP^1$. These results were obtained in order to assist with future proofs of Manin's conjecture
for del Pezzo surfaces with a conic bundle,
and to help with applications of Theorem \ref{thm:conic_bundle_precise}. 
The equations obtained here, together with the analytic tools in this paper,
have already found applications to Manin's conjecture for quartic del Pezzo surfaces \cite{BS16}.

\section{The geometry of conic bundles} \label{sec:conic_bundles_geometry}
In this section we study the geometry of conic bundle surfaces. We work over a fixed field $k$,
which for simplicity we assume satisfies $\chr(k) \neq 2$ in \S \ref{sec:conic_bundles_subsec}.
The primary purpose is to explain how to obtain equations
for arbitrary conic bundles inside $\PP^2$-bundles over $\PP^1$.

\subsection{Basic facts}

The following is well-known; we give proofs for completeness.

\begin{lemma} \label{lem:Pic}
	Let $\pi:X\to \PP^1$ be a conic bundle. Then
	\begin{enumerate}
		\item $\pi^{-1}(P)$ is reduced for all $P \in \PP^1$. \label{item:reduced}
		\item  \label{item:Picard}
		$\rank \Pic(X) = 2 + 
		\#\{ \text{closed points } P \in \PP^1 : \pi^{-1}(P) \text{ is singular and split} \}.$
	\end{enumerate}
\end{lemma}
\begin{proof}
	For the first part, we may assume that $k$ is algebraically closed.
	Recall from Definition~\ref{def:conic_bundle} that $X$ is smooth and projective over $k$.
	Suppose that $\pi^{-1}(P) = 2L$ is non-reduced.
	Then $(2L)^2 = L^2=0$. As $p_a(2L) = p_a(L)=0$, the adjunction formula 
	\cite[Ex.~V.1.3]{Har77} gives both $-K_X \cdot L = 2$
	and $-K_X \cdot 2L = 2$; contradiction.

	For the second part, let $P \in \PP^1$ be such that $\pi^{-1}(P)$ is split and singular.
	Then $\pi^{-1}(P)$ is the union of $2$ irreducible components, each of which consists of 
	a collection of pairwise skew $(-1)$-curves over $\bar{k}$. Blowing-down the choice
	of such an irreducible component and applying induction, we may assume that $\pi$ is relatively
	minimal, i.e.~each fibre $F$ of $\pi$ is irreducible. Consider the map
	$$\Pic X \to \ZZ$$
	given by intersecting each divisor with $F$. Its image has finite index, as we have $-K_X \cdot F = 2$ by the 
	adjunction formula. As the linear system $|F|$ is positive dimensional, we see that the kernel
	is generated by the irreducible components of the elements of $|F|$. But each element of $|F|$ is irreducible
	by assumption, and hence $\rank \Pic(X) = 2$ as required.
\end{proof}

\subsection{Projective bundles} \label{sec:proj}
 We recall some facts about projective bundles over $\PP^1$,
following Reid's treatment in \cite[\S2]{Rei93} (see \cite[p.~162]{Har77} for a more general approach).

\begin{definition}
Let $(a_0,\ldots,a_n) \in \ZZ^{n+1}$. We define the associated $\PP^n$-bundle over $\PP^1$ to be the projectivisation
$$\FF(a_0,\ldots,a_n) = \PP(\OO_{\PP^1}(a_0) \oplus \cdots \oplus \OO_{\PP^1}(a_n))$$
of the vector bundle $\OO_{\PP^1}(a_0) \oplus \cdots \oplus \OO_{\PP^1}(a_n)$ over $\PP^1$.
\end{definition}
As a special case, we have $\FF(0,\ldots,0) \cong \PP^1\times \PP^n$. 
One may view $\FF(a_0,\ldots,a_n)$ as a quotient
$$\FF(a_0,\ldots,a_n) \cong (\mathbb{A}^2\setminus 0) \times (\mathbb{A}^{n+1}\setminus 0) / \Gm^2,$$
where $(\lambda,\mu) \in \Gm^2$ acts via
\begin{equation} \label{eqn:Gm^2}
	(\lambda,\mu) \cdot (s,t;x_0,\ldots,x_n) = (\lambda s,\lambda  t;\lambda^{-a_0} \mu x_0,\ldots,\lambda^{-a_n} \mu x_n).
\end{equation}

We denote by $M,F \in \Pic \FF(a_0,\ldots,a_n)$ the class of the relative hyperplane bundle and 
the class of a fibre, respectively. These generate the Picard group \cite[Lem.~2.7]{Rei93}.

For $(e,d) \in \ZZ^2$, we say that a polynomial $f \in k[s,t,x_0,\ldots,x_n]$ is 
\emph{bihomogeneous of bidegree $(e,d)$} if it lies in the 
$\lambda^e \mu^d$-eigenspace for the action \eqref{eqn:Gm^2} of $\Gm^2$.
Such polynomials determine a well-defined subvariety $f(s,t;x_0,\ldots,x_n)=0$ of $\FF(a_0,\ldots,a_n)$.
A basis for this space is given by monomials of the form 
$s^{e_1}t^{e_2} x_0^{d_0}\cdots x_n^{d_n}$ where the exponents are non-negative and satisfy
$\sum_{i=0}^n d_i = d$ and $e_1 + e_2 = e + \sum_{i=0}^n d_ia_i$.
Such polynomials exactly correspond to the global sections of the sheaf $\OO(eF + dM)$.

\subsection{Conic bundles} \label{sec:conic_bundles_subsec}

The set-up of \S \ref{sec:proj} allows us to write the equations of conic bundles in an explicit manner as follows. 
For simplicity, in this section we assume that $\chr k \neq 2$.

A smooth surface of bidegree $(e,2)$ in $\FF(a_0,a_1,a_2)$ has an equation of the shape
\begin{equation} \label{eqn:conic_bundle}
	\sum_{0 \leq i,j \leq 2} f_{i,j}(s,t)x_ix_j= 0.
\end{equation}
Throughout this paper, we follow the convention that $f_{i,j} = f_{j,i}$.
One can determine the non-singularity of such a surface \eqref{eqn:conic_bundle} by applying the usual 
Jacobian criterion to the $6$ affine patches given by $sx_i\neq 0$ and $tx_i \neq 0$ for $i \in \{0,1,2\}$.

The degrees of the $f_{i,j}$ are given by the following matrix
\begin{equation} \label{eqn:degree_matrix}
\left( \begin{array}{ccc}
2a_0 + e & a_0 + a_1 + e & a_0 + a_2 + e \\
a_0 + a_1 + e & 2a_1 + e & a_1 + a_2 + e \\
a_0 + a_2 + e & a_1 + a_2 + e & 2a_2 + e \end{array} \right).
\end{equation}
Such a surface is equipped with the conic bundle structure $\pi: (s:t;x_0:x_1:x_2)  \mapsto (s:t)$.
The discriminant $\Delta_\pi(s,t) \in k[s,t]$ of $\pi$ is defined
to be the determinant of the matrix given by the quadratic form; the roots of $\Delta_\pi(s,t)$ correspond
to the singular fibres of $\pi$. 
As $X$ is smooth over $k$, the discriminant is a separable polynomial.
Note that the affine cone of \eqref{eqn:conic_bundle}
in $(\mathbb{A}^2\setminus 0) \times (\mathbb{A}^{3}\setminus 0)$ is a $\Gm^2$-torsor over $X$;
it is an example of a ``conic bundle torsor'', as used in \cite[\S2.2]{mikey}.

The following (well-known) lemma shows that every conic bundle has the above form.

\begin{lemma} \label{lem:conic_bundle_classification}
	Let $(a_0,a_1,a_2) \in \ZZ^3$ and $e \in \ZZ$. Then any smooth hypersurface of bidegree $(e,2)$ in 
	$\FF(a_0,a_1,a_2)$ is a conic bundle surface. Moreover, every conic bundle surface arises this way
	(for some choice of $(a_0,a_1,a_2) \in \ZZ^3$ and $e \in \ZZ$).
\end{lemma}
\begin{proof}
	The first part is clear. So let $\pi:X \to \PP^1$ be a conic bundle with 
	relative anticanonical bundle $\omega_\pi^{-1} = \omega_X^{-1}\otimes \pi^*(\omega_{\PP^1})$. 
	As explained in \cite[Cor.~3.7]{Has09} and its proof, the pushforward $\pi_* \omega_\pi^{-1}$
	is a locally free sheaf of rank $3$ and the map
	$$X \to \PP(\pi_* \omega_\pi^{-1})$$
	is a closed embedding, where $\PP(\pi_* \omega_\pi^{-1}) \to \PP^1$ denotes the associated $\PP^2$-bundle
	(this is shown in \emph{loc.~cit.} 
	under the additional assumption that $X$ is minimal, but the same proof works in our case).
	However, by a theorem of Birkhoff--Grothendieck \cite[Ex.~V.2.6]{Har77}, every
	vector bundle on $\PP^1$ is a direct sum of line bundles. Hence
	$\pi_* \omega_\pi^{-1} \cong \OO_{\PP^1}(a_0) \oplus \OO_{\PP^1}(a_1) \oplus \OO_{\PP^1}(a_2)$
	for some $a_i \in \ZZ$. As the relative hyperplane class induces the anticanonical class on each fibre,
	it is clear that $X$ has bidegree $(e,2)$ for some $e \in \ZZ$.
	The result follows.
\end{proof}

\begin{remark} \label{rem:not_unique}
Note that the triple $(a_0,a_1,a_2)$ and degree $e$ in Lemma \ref{lem:conic_bundle_classification}
are \emph{not} unique, due to the isomorphism $\FF(a_0,a_1,a_2) \cong \FF(a_0 + f,a_1 +f,a_2+ f)$ for any $f \in \ZZ$.
\end{remark}

We record  the required geometrical properties of such surfaces in the next proposition.
By abuse of notation, we denote by $M,F \in \Pic X$ the pull-back of the corresponding divisor classes 
from $\FF(a_0,a_1,a_2)$ to $X$.
See also \S \ref{sec:conics} for our conventions concerning the absolute values $| \cdot |_v$ and the local degrees $m_v$.

\begin{proposition} \label{prop:conic_bundle}
	Let $X \subset \FF(a_0,a_1,a_2)$ be a smooth surface of bidegree $(e,2)$. Then 
	\begin{enumerate} 
		\item $e \geq \max\{-2a_0,-2a_1,-2a_2\}$. \label{item:effective}
		\item $\deg \Delta_\pi(s,t) = 2(a_0 + a_1 + a_2) + 3e$. \label{item:Delta}
		\item $-K_X = M + (2 - a_0 - a_1 - a_2 - e)F$. \label{item:-K_X}
		\item \label{item:intersections}
		$ 
		F^2 = 0, \,\,\, M\cdot F = 2, \,\,\, M^2 = 2(a_0+a_1+a_2) + e, \,\,\, K_X^2 = 8 - \deg \Delta_\pi(s,t).
		$
		\item If $k$ is a number field, then a choice of anticanonical height function is given by \label{item:height}
		$$H_{-K_X}(s:t;x_0:x_1:x_2) = \prod_v \frac{\max_i\{\max\{|s|_v,|t|_v\}^{a_i}|x_i|_v\}^\locdegv}
		{\max\{|s|_v,|t|_v\}^{(a_0+a_1+a_2+e-2)\locdegv}}.$$
	\end{enumerate}
\end{proposition}
\begin{proof}
	Part \eqref{item:effective} simply says that $X$ defines an effective divisor.
	For part \eqref{item:Delta}, we note that $\deg \Delta_\pi$ is the trace of the matrix \eqref{eqn:degree_matrix}.
	Part \eqref{item:-K_X} follows from the adjunction formula \cite[Prop.~II.8.20]{Har77},
	on recalling that $K_{\FF(a_0,a_1,a_2)} = (-2 + a_0 + a_1 + a_2)F - 3M$ \cite[Rem.~p.~19]{Rei93}.
	
	As for part (\ref{item:intersections}), the calculations $F^2 = 0$ and $M\cdot F = 2$ are clear.
	The last equality follows from Noether's formula \cite[Ex.~A.4.1.2]{Har77},
	on noting from Lemma \ref{lem:Pic} that $\rho(\bar{S}) = 2 + \deg \Delta_\pi.$ The value for $M^2$
	can then be deduced from \eqref{item:Delta} and \eqref{item:-K_X}.
		
	For (\ref{item:height}), we recall some of the theory of height functions.
	Let $L$ be a line bundle on a projective variety $V$. 
	If there exists a collection of global sections $s_0,\ldots, s_n$ which generate $L$, then a choice
	of height function associated to $L$ is given by $$H_L(x) = \prod_v \max\{|s_0(x)|_v:\cdots : |s_n(x)|_v\}^\locdegv.$$
	When $L$ is not generated by global sections, it may be written as a difference $L \cong L_1 \otimes L_2^{-1}$
	where each $L_i$ is generated by global sections (see \cite[Thm.~B.3.2]{HS00} and its proof).
	In this case, one may take
	$$H_L(x) = \frac{H_{L_1}(x)}{H_{L_2}(x)}.$$
	This definition depends on the choice of the $L_i$ and the global sections, up to a bounded function.
	In our case, both $\OO_X(M)$ and $\OO_X(F)$ are generated by global sections. Choosing generating global sections,
	applying the height machine,
	and using \eqref{item:-K_X}, we obtain \eqref{item:height}.
\end{proof}

\section{Rational points on conics}   \label{sec:conics} 
\begin{notation}
In this and the next section, we work over a number field $K$ of degree $m$, with ring of integers $\OO_K$, and with algebraic closure $\overline{K}$. We choose a fixed set $\mathcal{C}$ of integral representatives for the ideal classes of $K$. All implicit constants are allowed to depend on $K$ and the choice of $\mathcal{C}$.

We denote the monoid of nonzero integral ideals of $\OO_K$ by $\ideals$ and the absolute norm of $\aaa\in\ideals$ by $\norm\aaa$. The letter $\ppp$ always denotes a nonzero prime ideal of $K$, and $v_\ppp$ the $\ppp$-adic exponential valuation on elements and ideals of $\OO_K$. We let $\F_\ppp = \OO_K/\ppp$. Euler's totient function for nonzero ideals of $\OO_K$ shall be denoted by $\phi_K$. For $a\in\OO_K$ and $\bbb = \ppp_1^{e_1}\cdots \ppp_l^{e_l}\in\ideals$, with distinct prime ideals $\ppp_i$ none of which lies above $2$, the \emph{Jacobi symbol} is defined as 
\begin{equation*}
  \qr{a}{\bbb} := \prod_{i=1}^l\qr{a}{\ppp_i}^{e_i},
\end{equation*}
where $\qrp{a}$ is the quadratic residue symbol for $K$.

 We write $\places,\archplaces,\nonarchplaces$ for the sets of all, all archimedean, and all nonarchimedean places of $K$, respectively. For $v\in\places$ lying above a place $w$ of $\QQ$, we write $\locdegv:=[K_v:\Q_w]$ for the local degree. We write $K_v$ for the completion of $K$ at $v$ and extend the usual (real or p-adic) absolute value of $\QQ_w$ to an absolute value  $\absv{\cdot}$ on $K_v$. All completions $K_v$ at archimedean places $v$ are identified with $\R$ or $\mathbb{C}$, and we identify the $\RR$-vector space $K_\infty := K\otimes_\Q\R = \prod_{v\mid\infty}K_v$ with $\R^{\dg}$ via the identification $\mathbb{C}\cong \R^2$. We include $K$ in $K_\infty$ via its natural embedding. The volume of a (measurable) subset of $K_\infty$ is its usual Lebesgue measure in $\R^\dg$.

We define the discriminant of a binary quadratic form $Q=ax_0^2+bx_0x_1+cx_1^2\in \OO_K[x_0,x_1]$ as $\Delta_Q:=b^2-4ac$, and the discriminant of a ternary quadratic form $Q\in\OO_K[x_0,x_1,x_2]$ as
\begin{equation*}
\Delta_Q := -\frac{1}{2}\det\left(\frac{\partial Q}{\partial x_i\partial x_j}\right)_{0\leq i,j\leq 2}.
\end{equation*}
These differ from the usual definition in terms of the associated bilinear form only by a factor $-4$ and ensure that $\Delta_Q\in\OO_K$. We denote the resultant of two binary forms $F,G\in\OO_K[s,t]$ by $\res(F,G)\in\OO_K$.

\end{notation}

\subsection{Counting rational points}
Let $K$ be a number field and $C \subset \P^2_K$ a smooth conic defined by a quadratic form 
\begin{equation}
Q(x_0,x_1,x_2) = ax_0^2+bx_0x_1+cx_1^2+dx_0x_2+ex_1x_2+fx_2^2\in\OO_K[x_0,x_1,x_2]. \label{eq:special form}
\end{equation}
We consider heights on $\P^2(K)$ of the following form. For every $v\in\places$, let $\vecnorm{\cdot}_{v,\max}$ be the $\max$-norm on $K_v^3$. For non-archimedean places $v$, we let
\begin{equation*}
  \locheightv{\cdot} := \vecnorm{\cdot}_{v,\max}.
\end{equation*}
For each archimedean place $v$, we fix an invertible linear transformation $A_v : K_v^3\to K_v^3$ and let
\begin{equation}\label{eq:loc_height_arch}
  \locheightv{\cdot} := \vecnorm{A_v(\cdot)}_{v,\max}.
\end{equation}
We then have a height function
\begin{equation}\label{eq:def height}
  H((x_0:x_1:x_2)) := \prod_{v\in\places}\locheightv{x_0,x_1,x_2}^{\locdegv}
\end{equation}
on $\P^2(K)$. We shall prove an asymptotic lower bound for the counting function
\begin{equation*}
  N_{C,H}(B) := \card\left\{\vx \in C(K) \where H(\vx)\leq B\right\}
\end{equation*}
that is explicit in terms of $Q$ and $H$. For every archimedean place $v$ of $K$, let 
\begin{equation*}
q_{H,v} := \max_{\vx\in  K_v^3\smallsetminus\{0\}}\frac{\max\{\absv{x_0},\absv{x_1},\absv{x_2}\}}{\locheightv{\vx}}\quad\text{ and }\quad q_H := \prod_{v\in\archplaces}q_{H,v}^\locdegv,
\end{equation*}
and moreover
\begin{equation*}
\qfheight{Q}_v := \max\{\absv{a},\absv{b},\absv{c},\absv{d},\absv{e},\absv{f}\}\quad\text{ and }\quad   \qfheight{Q} := 
\prod_{v\in\places}\qfheight{Q}_v^{\locdegv}.
\end{equation*}
We first consider the case where $c=0$, that is, $C$ has the rational point $(0:1:0)$. In this case, we obtain an asymptotic formula.
\begin{theorem}\label{thm:conics_special_form}
  Let $C\subset \P^2_K$ be the smooth conic given by the nondegenerate quadratic form $Q$ as in \eqref{eq:special form} with $c=0$. Let $H$ be a height function as above. Then
  \begin{equation*}
    N_{C,H}(B) = c_{C, H} B + O\left((q_H B)^{1-1/(2\dg)} \qfheight{Q}^{4} \mathcal{L}_1\right),
  \end{equation*}
  for $B \geq 1$. Here, $\mathcal{L}_1 = 1$ if $\dg > 1$ and
  $\mathcal{L}_1 = \log(q_H B)$ if $\dg=1$. The
  leading constant $c_{C,H}$ is positive and as described by Peyre \cite{peyre}, and the implied constant in the error term depends only on $K$.
\end{theorem}

For $K=\Q$ this follows from \cite[Prop.~2.1]{sofos}. The proof in \cite{sofos} can be adapted to arbitrary number fields using similar techniques to \cite{Marta}; we do not include the details in the interests of brevity. Since Manin's conjecture is already known to hold for $\P^1$, all the novelty of Theorem \ref{thm:conics_special_form} lies in the explicit error term. 

Let us briefly recall the shape of the constant $c_{C,H}$ as defined in \cite{peyre}. With a constant $c_K$ depending only on $K$, we have
\begin{equation}\label{eq:def_peyre_const}
  c_{C,H} = c_K \prod_{v\in\archplaces}\bomega_{H,v}(C(K_v))\prod_{v\in\nonarchplaces}\left(1-\frac{1}{q_v}\right)\bomega_{H,v}(C(K_v)).
\end{equation}
Here $\bomega_{H,v}$ is a measure on $C(K_v)$ that can, by \cite[Lemme 5.4.4]{peyre}, be described in local $v$-analytic coordinates, say, $\varphi : (1:x_1:x_2) \mapsto x_1$ by its pushforward
\begin{equation}\label{eq:density local coord}
  \varphi_*(\bomega_{H,v}) = \frac{\mathrm{d}x_1}{H_v(\varphi^{-1}(x_1))^\locdegv\absv{Q_{x_2}(\varphi^{-1}(x_1))}^\locdegv}.
\end{equation}
In the above formula, $\mathrm{d}x_1$ denotes the Haar measure on $K_v$, normalised by $\int_{\OO_v}\mathrm{d}x_1=1$ in the non-archimedean case, and the usual Lebesgue measure on $\RR$ or $\CC$ in the archimedean case.

To obtain a lower bound in the general case, we first recall that $C(K)\neq\emptyset$ already implies that $C$ has a $K$-point $\xi$ with $H_{\text{Weil}}(\xi)\ll \qfheight{Q}$, where $H_{\text{Weil}}$ is the usual Weil height obtained by taking the $\max$-norm at all places. This was originally proved over $\Q$ by Cassels \cite{casselsQF}, and first extended to number fields by Raghavan \cite{raghavan}. Given such a point $\xi = (\xi_0,\xi_1,\xi_2)$ with $\xi_j\in\OO_K$ and $\xi_0\OO_K + \xi_1\OO_K +\xi_2\OO_K \in \mathcal{C}$, we construct a matrix $A$ with entries in $\OO_K$ mapping $(0,1,0)$ to $\xi$. The proof in~\cite[Lemma 6.1]{sofos}
is directly adapted to our setting, thus providing the next result.

\begin{lemma}\label{lem:transition_matrix}
  Let $\xi_0,\xi_1,\xi_2 \in \OO_K$ with $\xi_0\OO_K + \xi_1\OO_K +\xi_2\OO_K \in \mathcal{C}$. Then there is a $3\times 3$-matrix $A = (a_{i,j})_{i,j}$ with entries $a_{i,j}\in\OO_K$, invertible over $K$, such that 
  \begin{equation*}
    \norm\det A \ll 1, \quad A\cdot (0,1,0)^t = (\xi_0,\xi_1,\xi_2)^t, \quad \text{and }
    \quad \prod_{v\in\archplaces}\max_{i,j}\{\absv{a_{i,j}}\}^{\locdegv} \ll H_{\text{Weil}}(\xi_0,\xi_1,\xi_2)^{\dg^2}.
  \end{equation*}
\end{lemma}
Assuming Theorem \ref{thm:conics_special_form} for a slightly more general class of heights, where the local factors may be defined as in \eqref{eq:loc_height_arch} also for a finite set of non-archimedean places, one could easily prove an asymptotic in the next proposition.
We do not pursue this here and are satisfied with an asymptotic lower bound.
\begin{proposition}\label{prop:conic_counting}
  Let $Q$ be a nondegenerate quadratic form as in \eqref{eq:special form}, $C\subset \P^2_K$ the corresponding smooth conic, and let $H$ be a height function as in \eqref{eq:def height}. Then 
  \begin{equation*}
    N_{C,H}(B) \gg c_{C,H}B + O(B^{1-\epsilon}(q_H\qfheight{Q})^{1/\epsilon}),
  \end{equation*}
  for some $\epsilon>0$ depending only on $\dg$. The constant $c_{C,H}$ is the one predicted by Peyre, and the implicit constants depend only on $K$.
\end{proposition}

\begin{proof}
  If $C(K)=\emptyset$ then, by Hasse-Minkowski, $C(K_v)=\emptyset$ for some $v\in\places$ and thus $c_{C,H}=0$. If $C(K)\neq\emptyset$, then we can find a point in $C(K)$ that has a representative $\xi=(\xi_0,\xi_1,\xi_2)\in\OO_K^3$ with $H_{\text{Weil}}(\xi)\ll \qfheight{Q}$ and $\xi_0\OO_K + \xi_1\OO_K + \xi_2\OO_K \in \mathcal{C}$. Let $A$ be a matrix as in Lemma \ref{lem:transition_matrix}, and define the height $\tilde{H}_A := \prod_{v\in\places}\tilde{H}_{A,v}^{\locdegv}$ by $\tilde{H}_{A,v}:=H_v$ for $v\in\archplaces$, and 
\begin{equation*}
\tilde{H}_{A,v}(\vx) := H_v(A^{-1}\vx)\geq H_v(\vx)
\end{equation*}
for $v\in\nonarchplaces$. Then clearly $H(\vx)\leq \tilde{H}_A(\vx)$, so
\begin{equation*}
  N_{C,H}(B)\geq N_{C,\tilde{H}_A}(B).
\end{equation*}
 Let $C_A\subset\P^2_K$ be the smooth conic defined by $Q(A\vx)=0$, and let $H_A := \tilde{H}_A\circ A$. Then $C_A$ has the rational point $(0:1:0)$, and
\begin{equation*}
  N_{C,\tilde{H}_A}(B) = N_{C_A,H_A}(B).
\end{equation*}
By Theorem \ref{thm:conics_special_form} we see that for sufficiently small $\varepsilon > 0$ we have
\begin{equation*}
 N_{C,\tilde{H}_A}(B) = c_{C_A,H_A}B + O(B^{1-\epsilon}(q_{H_A}\qfheight{Q(A\vx)})^{1/\epsilon}),
\end{equation*}
and in particular $c_{C_A,H_A} = c_{C,\tilde{H}_A}$.
Shrinking $\varepsilon$ and using the bounds 
\begin{equation*}
\prod_{v\in\archplaces}\max_{i,j}\{\absv{a_{i,j}}\}^\locdegv\ll \qfheight{Q}^{\dg^2}
\end{equation*}
derived from Lemma \ref{lem:transition_matrix} and our choice of $\xi$, we see that the error term is as desired. To finish our proof, we note from \eqref{eq:density local coord} and $H_{v}(A^{-1}\vx)\leq \absv{(\det A)^{-1}}H_v(\vx)$ that
\[
  c_{C,\tilde{H}_A} \geq 
\norm(\det A)^{-1}c_{C,H}. \qedhere
\]
\end{proof}

\subsection{Local densities of conics} \label{sec:local_densites_conics}
Let $C\subseteq \P^2_K$ be a smooth conic defined by a nondegenerate quadratic form $Q\in \OO_K[x_0,x_1,x_2]$. We are interested in lower bounds for the quantities
\begin{equation*}
  \sigma_\ppp(Q)=\sigma_v(Q) := \left(1-\frac{1}{q_v}\right)\bomega_{H,v}(C(K_v)),
\end{equation*}
for $v\in\nonarchplaces$ corresponding to prime ideals $\ppp\nmid 2$. Recall that the local factors of the height at these places are given by the max-norm. It is well known, and follows for example from \cite[Corollary 3.5]{MR1681100}, that in this case
\begin{equation*}
  \sigma_\ppp(Q) = \lim_{n\to\infty}\frac{\card\{(x_0,x_1,x_2)\bmod \ppp^n\where (x_0,x_1,x_2)\not\equiv \vz\bmod \ppp\text{ and } Q(x_0,x_1,x_2)\equiv 0 \bmod \ppp^n \}}{\norm\ppp^{2n}}.
\end{equation*}

Let $\Delta_Q\in\OO_K\smallsetminus\{0\}$ be the discriminant of $Q$. The following lemma is well known and easily proved by an application of Hensel's lemma,
on using the fact that $\#\PP^1(\OO_K/\ppp) = 1 + \norm \ppp$.

\begin{lemma}\label{lem:local_smooth}
For all prime ideals $\ppp\nmid 2\Delta_Q$, we have
  \begin{equation*}
    \sigma_\ppp(Q) = 1-\frac{1}{\norm\ppp^2}.
  \end{equation*}
\end{lemma}

\begin{lemma}\label{lem:M_p}
  Let $A,B \in \OO_K$ with $\ppp\nmid 2AB$, and define
  \begin{equation*}
    M_\ppp(A,B) := \card\{(x_0,x_1,x_2)\bmod \ppp \where x_2(x_0,x_1)\not\equiv\vz\bmod\ppp,\ Ax_0^2+Bx_1^2-x_2^2\equiv\vz\bmod\ppp\}.
  \end{equation*}
  Then
  \begin{equation*}
    M_\ppp(A,B) = \left(\norm\ppp-1\right)\left(\norm\ppp-\qrp{-AB}\right).
  \end{equation*}
\end{lemma}

\begin{proof}
	It suffices to show that 
	\[
	\# \{ (x_0:x_1:x_2) \in \PP^2(\FF_\ppp): x_2 \neq 0, \, Ax_0^2+Bx_1^2 = x_2^2\} = \norm\ppp-\qrp{-AB}.
	\]
	The equation $Ax_0^2+Bx_1^2 = x_2^2$ defines a smooth plane conic over $\FF_\ppp$, hence has $\norm\ppp + 1$
	rational points. Exactly $1+\qrp{-AB}$ of these rational points satisfy $x_2=0$, whence the result.
\end{proof}

Now assume that $\ppp\nmid 2$, that $\ppp\mid\Delta_Q$, and that the rank of $Q$ over $\F_\ppp$ is $2$. Let $C_\ppp$ be the singular conic defined by $Q$ over $\F_\ppp$, and
\begin{equation}\label{eq:def_chi_p}
  \chi_\ppp(Q) :=
  \begin{cases}
    1 & \text{ if $C_\ppp$ is split}\\
    -1 & \text{ if $C_\ppp$ is non-split.}
  \end{cases}
\end{equation}
Let $\OO_\ppp$ be the completion of $\OO_K$. It is well known that one can diagonalize $Q$ over $\OO_\ppp$ \cite[Cor.~I.3.4]{MH73}. Thus, we may assume that $Q$ is equivalent over $\OO_\ppp$ to the form
\begin{equation}\label{eq:conic_diag}
  ax_0^2 + bx_1^2 - \pi^{v_\ppp(\Delta_Q)}x_2^2,
\end{equation}
with $\ppp$-adic units $a,b$ and a uniformiser $\pi$ at $\ppp$. With the equation  \eqref{eq:conic_diag}, we clearly have
\begin{equation}\label{eq:chi p}
  \chi_\ppp(Q) = \qrp{-ab}.
\end{equation}

The lower bound for $\sigma_\ppp(Q)$ in the next proposition is crucial for the passage from rational points on conic bundles to divisor sums in the proof of Theorem \ref{thm:conic bundle}.

\begin{proposition}\label{prop:local density bounds}
  Let $Q \in \OO_K[x_0,x_1,x_2]$ be a quadratic form that is nondegenerate over $K$, and assume that $\ppp\nmid 2$, that $\ppp \mid \Delta_Q$, and that the rank of $Q$ over $\F_\ppp$ is $2$. Then
\begin{equation*} \left(1-\frac{2}{\norm\ppp}\right)\sum_{k=0}^{v_\ppp(\Delta_Q)}\chi_\ppp(Q)^k \leq \sigma_\ppp(Q) \leq \sum_{k=0}^{v_\ppp(\Delta_Q)}\chi_\ppp(Q)^k.
\end{equation*}
\end{proposition}

\begin{proof}
  Let $n \geq v_\ppp(\Delta_Q)+1$ and define
  \begin{equation*}
    N^*(\ppp^n) := \card\{(x_0,x_1,x_2)\bmod\ppp^n\where (x_0,x_1,x_2)\not\equiv\vz\bmod\ppp \text{ and } Q(x_0,x_1,x_2)\equiv 0\bmod \ppp^n\}.
  \end{equation*}
  We will evaluate $N^*(\ppp^n)$ explicitly.  Diagonalising over $\OO_\ppp$, we may assume that $Q$ has the form \eqref{eq:conic_diag} modulo $\ppp^n$, with $a,b \in \OO_K$ units modulo $\ppp$ and $\pi\in\OO_K$ with $v_\ppp(\pi)=1$. Let
\begin{equation*}
  N_0^*(\ppp^n) := \card\{(x_0,x_1,x_2)\bmod\ppp^n\where Q(x_0,x_1,x_2)\equiv 0\bmod\ppp^n \text{ and }(x_0,x_1)\not\equiv\vz\bmod\ppp \}
\end{equation*}
and, for $0<d\leq n$,
\begin{equation*}
  N_d^*(\ppp^n) := \card\{(x_0,x_1,x_2)\bmod\ppp^n\where Q(x_0,x_1,x_2)\equiv 0\bmod \ppp^n,\ \min\{v_\ppp(x_0),v_\ppp(x_1)\}=d,\ \ppp\nmid x_2\}.
\end{equation*}
We can then write
\begin{equation}\label{eq:n_star_sum}
  N^*(\ppp^n) = \sum_{d=0}^nN_d^*(\ppp^n).
\end{equation}
For $d > v_\ppp(\Delta_Q)/2$ we have $\ppp^{1+v_\ppp(\Delta_Q)} \mid ax_0^2 + bx_1^2$, which is impossible if $\ppp \nmid x_2$, whence $N_d^*(\ppp^n) = 0$. By Hensel's Lemma,
\begin{equation*}
  \frac{N_0^*(\ppp^n)}{\norm\ppp^{2n}} = \frac{N_0^*(\ppp)}{\norm\ppp^2}=\left(1-\frac{1}{\norm\ppp} \right)\left(1+\qrp{-ab}\right) = \left(1-\frac{1}{\norm\ppp}\right)\left(1+\chi_\ppp(Q)\right).
\end{equation*}

Next, we evaluate $N_d^*(\ppp^n)$ for $0<d\leq v_\ppp(\Delta_Q)/2$. Here, write $(x_0,x_1)\equiv \pi^d(x_0',x_1')\bmod\ppp^n$, with $(x_0',x_1')\in\OO_K/\ppp^{n-d}$ and $(x_0',x_1')\not\equiv\vz\bmod\ppp$, so
\begin{equation*}
  N_d^*(\ppp^n) = \card\left\{(x_0',x_1')\bmod \ppp^{n-d},\ x_2\bmod \ppp^n\WHERE
  \begin{aligned}
    &ax_0'^2 + bx_1'^2 - \pi^{v_\ppp(\Delta_Q)-2d}x_2^2 \equiv 0 \bmod
    \ppp^{n-2d}\\
    &(x_0',x_1')\not\equiv\vz\bmod\ppp \text{ and }x_2\not\equiv 0\bmod \ppp
  \end{aligned}
\right\}.
\end{equation*}
This is equal to $\norm\ppp^{4d}M_{v_\ppp(\Delta_Q)-2d}(\ppp^{n-2d})$, where we defined
\begin{equation*}
  M_c(\ppp^k) := \card\{(x_0,x_1,x_2)\bmod \ppp^k \where ax_0^2+bx_1^2-\pi^cx_2^2 \equiv 0 \bmod \ppp^k,\ x_2(x_0,x_1)\not\equiv\vz\bmod\ppp\}.
\end{equation*}
By Hensel's Lemma,
\begin{equation*}
  \frac{M_c(\ppp^k)}{\norm\ppp^{2k}} = \frac{M_c(\ppp)}{\norm\ppp^2}
\end{equation*}
holds for all $k\geq 1$, and for $c\geq 1$ we have, similarly as before,
\begin{equation*}
  \frac{M_c(\ppp)}{\norm\ppp^2} = \left(1-\frac{1}{\norm\ppp}\right)^2(1+\chi_\ppp(Q)).
\end{equation*}
Lemma \ref{lem:M_p}, yields the equality
\begin{equation*}
  \frac{M_0(\ppp)}{\norm\ppp^2} = \left(1-\frac{1}{\norm\ppp}\right)\left(1-\frac{\chi_\ppp(Q)}{\norm\ppp}\right).
\end{equation*}
From \eqref{eq:n_star_sum}, we obtain that
\begin{equation}\label{eq:N star sum}
  \frac{N^*(\ppp^n)}{\norm\ppp^{2n}} = \frac{N_0^*(\ppp^n)}{\norm\ppp^{2n}} + \sum_{d=1}^{\lfloor v_\ppp(\Delta_Q)/2\rfloor}\frac{M_{v_\ppp(\Delta_Q)-2d}(\ppp^{n-2d})}{\norm\ppp^{2n-4d}}.
\end{equation}
Combining this with our previous computations, one easily finds that \eqref{eq:N star sum} equals
\begin{equation*}
	\begin{cases}
	 \left(1-\frac{1}{\norm\ppp}\right)\left(\left(1+\chi_\ppp(Q)\right)+\frac{v_\ppp(\Delta_Q)-1}{2}\left(1-\frac{1}{\norm\ppp}\right)(1+\chi_\ppp(Q))\right), 
	 & v_\ppp(\Delta_Q) \notin 2\ZZ, \\
	 \left(1-\frac{1}{\norm\ppp}\right)\left( (1 + \chi_\ppp(Q)) + \frac{v_\ppp(\Delta_Q)-2}{2}\left(1-\frac{1}{\norm\ppp}\right)(1+\chi_\ppp(Q))
	 + \left(1-\frac{\chi_\ppp(Q)}{\norm\ppp}\right)\right),
	 & v_\ppp(\Delta_Q) \in 2\ZZ.
	 \end{cases}
\end{equation*}
In both cases this is independent of $n$, hence determines $\sigma_\ppp(Q) = \lim_{n \to \infty} N^*(\ppp^n)/\norm\ppp^{2n}$. Distinguishing between the cases $\chi_\ppp(Q) = 1$ and $\chi_\ppp(Q) = -1$, it is easily verified that these expressions satisfy the required bounds.
\end{proof}

\section{Rational points on conic bundles}\label{sec:conic bundle counting}

The goal of this section is to prove Theorem \ref{thm:conic bundle} (in fact it will follow from a more precise version below). Let $\pi:X\to \PP^1_K$ be a conic bundle. By Lemma \ref{lem:conic_bundle_classification} and Remark \ref{rem:not_unique}, we may assume that $X$ is a smooth hypersurface of bidegree $(e,2)$ in $\FF(a_0,a_1,a_2)$, with $a_0=0$ and $e,a_1,a_2\geq 0$. Hence, let $X$ be cut out by the bihomogeneous form
\begin{equation}\label{eq:def form Q}
  Q_{(s,t)}(\vx) = Q(s,t,x_0,x_1,x_2) := \sum_{0\leq i,j\leq 2}f_{i,j}(s,t)x_ix_j \in \OO_K[s,t,x_0,x_1,x_2]
\end{equation}
of bidegree $(e,2)$, with $\deg f_{i,j} =a_i+a_j+e$. In this paper we are only concerned with bounds, and not an asymptotic formula.
In particular, by standard properties of height functions, from Proposition \ref{prop:conic_bundle} we may assume that the anticanonical height on $X(K)$ is given by $H(s,t,\vx) = \prod_{v\in\places}H_v(s,t,\vx)^{\locdegv}$, with $H_v(s,t,\vx)$ defined as
\begin{equation}\label{eq:local height shape}
\max\{\absv{s},\absv{t}\}^{2-a_1-a_2-e}\max\{\absv{x_0},\max\{\absv{s},\absv{t}\}^{a_1}\absv{x_1},\max\{\absv{s},\absv{t}\}^{a_2}\absv{x_2}\}.
\end{equation}
We write
\begin{equation*}
\Delta(s,t)=-4 \Delta_\pi(s,t)\in\OO_K[s,t],
\end{equation*}
where $\Delta_\pi$ is the discriminant of the conic bundle morphism $\pi:(s:t;\vx)\mapsto (s:t)$, and we denote the degree of $\Delta(s,t)$ by $d=2a_1+2a_2+3e$. A linear change of the variables $s,t$, which only affects heights by a bounded amount, allows us to assume moreover that $t\nmid \Delta(s,t)$.

Let us note that the bihomogeneity of $Q$ implies in particular that, for $\lambda\in K\smallsetminus\{0\}$, 
\begin{equation}\label{eq:homogeneity condition}
  Q(\lambda s,\lambda t,\vx) = \lambda^{e}Q(s,t,x_0,\lambda^{a_1}x_1,\lambda^{a_2}x_2).
\end{equation}
\subsection{Conic bundles and divisor sums}
\label{sec:conic bundles and divisors}
In the companion paper \cite{divsumpub}, the first and last named authors introduce certain generalised divisor sums over the values of binary forms and state a lower bound conjecture for such divisor sums. 
We construct now an instance of \cite[Conjecture 1]{divsumpub}, whose validity will imply the validity of Theorem \ref{thm:conic bundle} for the particular hypersurface $X$ of bidegree $(e,2)$ in $\FF(0,a_1,a_2)$ as above. 

First, we construct a system $\mathcal{F}(X)$ of forms as follows: let 
$M$ be the set of all closed points $p\in\P^1_K$, such that the fibre $X_p$ is singular.
 There is a one-to-one correspondence between points $p\in M$ and irreducible factors $\Delta_p(s,t)$ of $\Delta(s,t)$ in $K[s,t]$, up to multiplication by non-zero constants. Thus, we have a factorisation
\begin{equation}\label{eq:disc_fact}
  a\Delta(s,t) = \prod_{p\in M}\Delta_p(s,t),
\end{equation}
where $a\in \OO_K\smallsetminus \{0\}$ and $\Delta_p(s,t) \in \OO_K[s,t]$ is a form irreducible over $K$ with $\Delta_p(1,0)\neq 0$. The residue field $K(p)$, for $p\in M$, is generated over $K$ by a root $\theta_p$ of the polynomial $\Delta_p(s,1)$, and we may assume that $\theta_p\in\Kbar$. 

For any $p \in M$, we define a binary form $\delta_p(s,t)\in\OO_K[s,t]$ as follows: the plane conic $C_{(\theta_p,1)}$ defined by the ternary quadratic form $Q_{(\theta_p,1)}\in K(p)[s,t]$ is isomorphic to the fibre $X_p$ over $K(p)$, and thus singular. Let $i_p\in\{0,1,2\}$ be the index of the first non-zero coordinate of the unique singular point of $C_p$. Then we define $\delta_p(s,t)$ as the discriminant of the binary quadratic form obtained by setting $x_{i_p}:=0$ in $Q_{(s,t)}(x_0,x_1,x_2)$. We observe directly from the definition of $Q$ that $\delta_p$ is a binary form of even degree. We will prove in Lemma \ref{lem:ref_T_diagonalization} that $\delta_p(\theta_p,1)$ is a square in $K(p)$ if and only if the fibre $X_p$ is split. 

We define the system $\mathfrak{F}(X)$ as 
\begin{equation*}
  \mathfrak{F}(X) := \{(\Delta_p,\delta_p)\where p\in M\}.
\end{equation*}
Moreover, we define the multiplicative function $f:\ideals\to (0,\infty)$ by $f(\ppp) := -2/\norm\ppp$ for any prime ideal $\ppp$ not above $2$, and $f(\ppp^e):=0$ for all other prime ideal powers. 

The rest of this section is devoted to the proof of the following result.

\begin{theorem}\label{thm:conic_bundle_precise}
  Let $X$ be a smooth hypersurface of bidegree $(e,2)$ in $\FF(0,a_1,a_2)$, with $e,a_1,a_2\geq 0$, cut out by the form $Q$ as in \eqref{eq:def form Q}. Let $H$ be the anticanonical height on $X$ given by \eqref{eq:local height shape}. Assume that there is $(s_0,t_0)\in\OO_K^2\smallsetminus\{0\}$, such that the fibre of $\pi$ above $(s_0:t_0)\in\PP^1(K)$ is smooth and has rational points, and define the ideal $\classrep := s_0\OO_K + t_0\OO_K$. 

If the lower bound conjecture for divisor sums \cite[Conjecture 1]{divsumpub} holds for $K,\classrep,f,\mathfrak{F}(X)$, then 
  \begin{equation*}
    N_{U,H}(B)\gg B(\log B)^{\rho(X) - 1}
  \end{equation*}
  holds for every non-empty open subset $U$ of $X$. The implicit constant may depend on every parameter except $B$.
\end{theorem}

Theorem \ref{thm:conic_bundle_precise} implies Theorem \ref{thm:conic bundle}, since the complexity (as defined in \cite{divsumpub}) of $\mathfrak{F}(X)$ is exactly the complexity of the conic bundle $\pi:X\to\PP^1_K$, as will follow from Lemma \ref{lem:ref_T_diagonalization}.

We start our proof of Theorem \ref{thm:conic_bundle_precise} by constructing the divisor sums to which the lower bound conjecture will be applied. Let $\DD\subset K_\infty^2$ be a set of the form 
\[\DD=\prod_{v|\infty} \DD_v,
\]
where  $\DD_v\subseteq K_v^2$ is a compact ball of positive radius. Let $\www$ be an ideal of $\OO_K$ divisible by $2\classrep$, and let $\sigma,\tau\in\OO_K$ with $\sigma\OO_K+\tau\OO_K = \classrep$. For the triplet 
\[\c{P}:=(\c{D},(\sigma,\tau),\www),\]
and each $B\geq 1$, we define the sets
\[  M^*(\c{P},B):=
\{(s,t)\in\classrep^2\cap B^{1/m}\DD \where (s,t)\equiv (\sigma,\tau)\bmod \www
\text{, }s\OO_K + t\OO_K = \classrep\}
\]
and
\[
  M^*(\c{P},\infty) := \bigcup_{B\geq 1}M^*(\c{P},B).
\]
For any $\aaa\in\ideals$, we write 
\begin{equation}\label{flats}
  \aaa^\flat := \prod_{\substack{\ppp\nmid \www}}\ppp^{v_\ppp(\aaa)},
\end{equation}
and for $a\in\OO_K\smallsetminus\{0\}$, let $a^\flat := (a\OO_K)^\flat$. Observe that this notation depends on $\www$. 

As in \cite{divsumpub}, we call the triplet $\c P$ \emph{admissible}, if 
the following conditions~\eqref{eq:f not zero}-\eqref{eq:qr is one} hold:
\begin{equation}
\label{eq:f not zero}
\Delta_p(\sigma,\tau)\neq 0 \
\text{ for all } p\in M,
\end{equation} 
and whenever $(s,t)\in M^*(\c{P},\infty)$,
we have
\begin{equation}
\label{eq:D not zero}
\Delta_p(s,t)\neq 0 \ \text{ for all } p\in M,
\end{equation}
as well as 
\begin{equation}
\label{eq:qr is one}
\qr{\delta_p(s,t)}{\Delta_p(s,t)^\flat} = 1
\
\text{ for all } p\in M
.\end{equation}
With the function $\one_f:\ideals\to (0,\infty)$ given by
\begin{equation*}
  \one_f(\aaa) = \prod_{\ppp\mid\aaa}(1+f(\ppp)) = \prod_{\substack{\ppp\mid\aaa\\\ppp\nmid 2}}\left(1-\frac{2}{\norm\ppp}\right),
\end{equation*}
we define the function $r:M^*(\c{P},\infty)\to [0,\infty)$ by
\[
r(s,t) = r(\mathfrak{F}(X),f,\c{P};s,t)
:=
\prod_{p\in M}
\one_{f}(\Delta_p(s,t)^\flat)
\l(\sum_{\ddd_i|\Delta_p(s,t)^{\flat}}\qr{\delta_p(s,t)}{\ddd_i}\r)
\]
and consider the sum
\begin{equation*}
D(\mathfrak{F}(X),f,\c{P};B)
:=
\sum_{\substack{(s,t) \in M^*(\c{P},B)}}
r(\mathfrak{F}(X),f,\c P;s,t).
\end{equation*}

This is a generalized divisor sum, as in \cite{divsumpub}. Assuming the validity of \cite[Conjecture 1]{divsumpub} for $K$, $\classrep$, $f$ and $\mathfrak{F}(X)$, we obtain a finite set $S_{\text{bad}}$ of prime ideals of $\OO_K$, such that
\begin{equation}\label{eq:div sum bound}
  D(\mathfrak{F}(X),f,\c{P};B) \gg B^2(\log B)^{\card\{p\in M\where X_p \text{ split}\}},
\end{equation}
whenever $\www$ is divisible by each prime in $S_{\text{bad}}$ and the triplet $\c P$ is admissible. 

In the rest of this section, we reduce Theorem \ref{thm:conic_bundle_precise} to the statement \eqref{eq:div sum bound}, for a suitable
choice of admissible triplet $\c P$.

\subsection{Partitioning the rational points into fibres}\label{sec:partitioning}
We note that the fibres of the conic bundle morphism $\pi:X\to\P^1_K$ are isomorphic to plane conics as follows: let $L\supset K$ be an extension field. For $(s,t)\in L^2\smallsetminus\{0\}$, define the plane conic
\begin{equation*}
C_{(s,t)}\subset \P^2_L \quad\text{ by }\quad Q_{(s,t)}(\vx) = 0.  
\end{equation*}
For $p=(s:t)\in \P^1_K(L)$, we have then an isomorphism $\phi_{(s,t)}:C_{(s,t)} \to X_{p}$, given by sending $\vx$ to $(s:t;x_0:x_1:x_2) \in \FF(0,a_1,a_2)$, for any representative $(x_0,x_1,x_2)\in L^3\smallsetminus\{0\}$ of $\vx$. 

Sorting $X(K)$ by fibres, we get
\begin{equation*}
  N_{U,H}(B) = \sum_{(s:t)\in \P^1(K)}\card\{\vx \in (X_{(s:t)}\cap U)(K) \where H(\vx)\leq B\}.
\end{equation*}

Let $\freeunits$ be the subgroup of $\OO_K^\times$ generated by a fixed set of fundamental units. Then $\OO_K^\times = \freeunits \oplus \mu_K$, where $\mu_K$ is the group of roots of unity in $K$. We now construct a fundamental domain for the diagonal action of $\freeunits$ on $K^2\smallsetminus\{0\}$, as in \cite{MR2247898}. We consider $K$ as a subfield of $K_\infty = K\otimes_{\Q}\R = \prod_{v\in\archplaces}K_v$ via the natural embedding sending an element to its conjugates. Let $l:\prod_{v\in\archplaces}K_v^\times \to \R^{\archplaces}$ be the logarithmic map $(x_v)_v\mapsto (\locdegv\log\absv{x_v})_v$. Then $l(\freeunits)=l(\OO_K^\times)$ is a lattice in the hyperplane $\Sigma\subset \R^\archplaces$ given by $\sum_v y_v=0$. Let $F$ be a fundamental parallelotope for the lattice $l(\OO_K)$, and $\delta := (\locdegv)_{v}\in\R^\archplaces$.
Let 
\begin{align}\label{eq:def_fund_dom}
  \mathcal{G} &:= \left\{(s_v,t_v)_v\in\prod_{v\in\archplaces}(K_v^2\smallsetminus\{0\}) \where (\locdegv\log \max\{\absv{s_v},\absv{t_v}\})_v \in F+\R\delta\right\}.
\end{align}
Then $\mathcal{G}\cap K^2$ is a fundamental domain for the diagonal action of $\freeunits$ on $K^2\smallsetminus\{0\}$. Moreover, $\mathcal{G}$ is a cone, that is, $t\mathcal{G} = \mathcal{G}$ for all $t>0$. Multiplying $(s_0,t_0)$ by a unit and slightly shifting the fundamental parallelotope $F$, if necessary, we can ensure that $(s_0,t_0)$ lies in the interior of $\mathcal{G}$. Define the set
\begin{equation} \label{def:B(B)}
  \mathcal{B}(B) := \left\{(s,t)\in \classrep^2 \cap \mathcal{G}\WHERE
    \begin{aligned}
      &H((s:t))\leq B,\\
      &s\OO_K + t\OO_K = \classrep,\\
      &\Delta(s,t)\neq 0
    \end{aligned}
\right\}.
\end{equation}

\subsection{Lower bound for the counting problem}
Since every rational point in $\P^1(K)$ has at most $|\mu_K|$ representatives $(s,t) \in \mathcal{B}(B^\delta)$, for every $\delta>0$, we obtain the lower bound
\begin{align*}
  N_{U,H}(B) &\gg \sum_{(s,t)\in\mathcal{B}(B^\delta)}\card\{\vx\in(X_{(s:t)}\cap U)(K)\where H(\vx)\leq B\}.
\end{align*}

Let $V\subset X$ be the closed subvariety complement to $U$. It has finitely many irreducible components, with each one being either a point or a divisor of vertical or horizontal type. 

Vertical components are irreducible components of the fibres of $\pi$, and the height $H$ induces an anticanonial height on each smooth fibre
(recall from \eqref{def:B(B)} that we are not counting
rational points on singular fibres). Since Manin's conjecture is known to hold for $\PP^1$, there are $O(B)$ points of height at most $B$ in each smooth fibre.

Horizontal components meet each fibre in the same number of points, counted with intersection multiplicities, and, by Schanuel's theorem, there are $O(B^{2\delta})$ fibres $X_{(s:t)}$ with $(s,t)\in\mathcal{B}(B^{\delta})$. Thus, 
\begin{align*}
N_{U,H}(B)&\gg \sum_{(s,t)\in\mathcal{B}(B^\delta)}\card\{\vx\in X_{(s:t)}(K)\where H(\vx)\leq B\} + O(B+B^{2\delta}+1).
\end{align*}

Recall the definition of our height $H$ on $X(K)$ in \eqref{eq:local height shape}. Define on $C_{(s,t)}(K)$ the height $H_{(s,t)}(\vx)=\prod_{v\in\places}H_{(s,t),v}(\vx)^{\locdegv}$, where 
\begin{equation*}
  H_{(s,t),v}(\vx) := (H_v\circ\phi_{(s,t)})(\vx),\quad\text{ for }v\in\archplaces,
\end{equation*}
and $H_{(s,t),v}(\vx) := \max\{\absv{x_0},\absv{x_1},\absv{x_2}\}$ for $v\in\nonarchplaces$. Let $b:=2-a_1-a_2-e$. Since $a_1,a_2\geq 0$, we have, for $s,t\in \classrep$ with $s\OO_K+t\OO_K=\classrep$, the inequality
\begin{equation*}
   (H_v\circ\phi_{(s,t)})(\vx) \leq \absv{\classrep}^{b} H_{(s,t),v}(\vx)
\end{equation*}
for all $v\in\nonarchplaces$, whence $H\circ\phi_{(s,t)}\leq \norm\classrep^{-b} H_{(s,t)}$ as well.
Let
\begin{equation*}
  N_{C_{(s,t)}, H_{(s,t)}}(B) := \card\{\vx \in C_{(s,t)}(K)\where H_{(s,t)}(\vx)\leq B\}.
\end{equation*}
Then
\begin{equation}\label{eq:estimated by conics}
  N_{U,H}(B) \gg \sum_{(s,t)\in\mathcal{B}(B^\delta)}N_{C_{(s,t)},H_{(s,t)}}(\norm\classrep^{b}B) + O(B+B^{2\delta}).
\end{equation}

We introduce the short notation $\bomega_v(s,t):=\bomega_{H_{(s,t),v}}(C_{(s,t)}(K_v))$ and
\begin{equation*}
c(s,t):= \prod_{v\in\archplaces}\bomega_v(s,t)\prod_{v\in\nonarchplaces}\left(1-\frac{1}{q_v}\right)\bomega_v(s,t).
\end{equation*}
Let
\begin{equation*}
  \mathfrak{S}(B):=\sum_{(s,t)\in\mathcal{B}(B)}c(s,t).
\end{equation*}
Using Proposition \ref{prop:conic_counting} to estimate $N_{C_{(s,t)},H_{(s,t)}}(B)$ in \eqref{eq:estimated by conics} from below, we obtain, for small enough $\delta>0$,
\begin{equation*}
  N_{U,H}(B)\gg \mathfrak{S}(B^\delta) B + O(B).
\end{equation*}
Using Lemma \ref{lem:Pic}, we see that to prove Theorem \ref{thm:conic_bundle_precise} it suffices to show the bound
\begin{equation}\label{eq:desired bound 1}
  \mathfrak{S}(B)\gg (\log B)^{\card\{p \in M\where X_p \text{ split}\}+1}.
\end{equation}

\subsection{Local density in families}
Let $F=K_v$ be the completion of $K$ at any place $v$, with corresponding absolute value $\abs{\cdot}=\absv{\cdot}$. Let $\mathcal{U}\subset F^2$ be any open subset satisfying $\Delta(s,t)\neq 0$ for all $(s,t)\in \mathcal{U}$. Then
\begin{equation*}
  \mathcal{M}_{\mathcal{U}} := \{(s,t,\vx)\in \mathcal{U} \times \P^2(F)\where Q(s,t,\vx)=0\}\subseteq \mathcal{U} \times \P^2(F)
\end{equation*}
is an analytic manifold, as follows from the analytic version of the implicit function theorem (we follow the conventions of \cite[\S II.III]{Ser06} concerning analytic manifolds over local fields). Let $(\sigma,\tau)\in F^2\smallsetminus\{0\}$ and $\vxi\in \P^2(F)$ with $Q(\sigma,\tau,\vxi)=0$ and $\Delta(\sigma,\tau)\neq 0$. 

\begin{lemma}\label{lem:trivial_fibre_bundle}
  There is an open neighbourhood $\mathcal{U}_0$ of $(\sigma,\tau)$ such that the
projection
   \begin{equation*}
     \alpha:\mathcal{M}_{\mathcal{U}_0} \to \mathcal{U}_0
  \end{equation*}
  is a proper surjective analytic map.
\end{lemma}

\begin{proof}
	Let $V \subset \mathbb{A}^2$ be the complement of $\Delta(s,t) = 0$.
	Consider the morphism of varieties 
	\begin{equation} \label{eqn:morphism}
		\{(s,t,\vx)\in V \times \P_F^2 \where Q(s,t,\vx)=0\} \to  V
	\end{equation}
	over $F$. This morphism is projective, hence the induced map
	on analytic manifolds is a proper morphism (see \cite[p.~110]{Sal96} for this claim). Moreover the morphism \eqref{eqn:morphism} is smooth,
	hence the induced map of analytic manifolds is a submersion \cite[Prop.~2.7]{Sal96}, 
	in particular, it admits a section in some open neighbourhood $\mathcal{U}_0$ of $(\sigma,\tau)$
	(this is one of the equivalent definitions for a map to be a submersion; see \cite[II.III.10]{Ser06}).
	This completes the proof.
\end{proof}

We now prove the following fact about Peyre's local density $\bomega_v(s,t)=\bomega_{H_{(s,t),v}}(C_{(s,t)}(F))$.

\begin{lemma}\label{lem:local density continuous}
  The map $(s,t)\mapsto \bomega_v(s,t)$ is continuous on $\mathcal{U}_0$.
\end{lemma}

\begin{proof}
  It is enough to prove continuity on any open subset $\mathcal{U}_1$ whose closure $\overline{\mathcal{U}}_1$ is contained in $\mathcal{U}_0$. Since $\Delta(s,t)\neq 0$ on $\mathcal{U}_0$, we may cover $\mathcal{M}_{\mathcal{U}_0}$ by charts $\varphi_i : V_i \to W_i\subset F^{3}$ of the form, say, $(s,t,(1:x_1:x_2)) \mapsto (s,t,x_1)$, with $W_i$ a small ball with respect to the max-norm on $F^3$.

Since $\alpha^{-1}(\overline{\mathcal{U}}_1)$ is compact, finitely many of these charts are enough to cover $\alpha^{-1}(\mathcal{U}_1)$, say $\varphi_1, \ldots, \varphi_k$, and moreover we may find a corresponding partition of unity $f_1, \ldots, f_k$ of $\alpha^{-1}(\mathcal{U}_1)$ with $\overline{\supp(f_i)}\subseteq V_i$. Then
\begin{equation*}
  \bomega_v(s,t) = \sum_{i=1}^k\int_{(s,t,x_1)\in W_i}\frac{f_i(\varphi_i^{-1}(s,t,x_1))\mathrm{d}x_1}{H_v(\varphi_i^{-1}(s,t,x_1))^\locdegv\absv{Q_{x_2}(\varphi_i^{-1}(s,t,x_1))}^\locdegv},
\end{equation*}
with $H_v(s,t,\vx) := H_{(s,t),v}(\vx)$. However, each integrand is continuous in $s,t,x_1$ and compactly supported, so it can be bounded from above by a constant.
Hence $\bomega_v(s,t)$ is continuous on $\mathcal{U}_1$ by Lebesgue's dominated convergence theorem.
\end{proof}

Recall that we are given $(s_0,t_0)\in \OO_K^2$ with $s_0\OO_K + t_0\OO_K = \classrep$, such that $C_{(s_0,t_0)}(K)\neq \emptyset$ and $\Delta(s_0,t_0)\neq 0$. We note some consequences of the previous lemmata.

\begin{corollary}\label{cor:local density nonarch bound}
  Let $\www\in\ideals$ with $\classrep\mid\www$. Then there is $l\in\N$ such that
  \begin{equation*}
    \prod_{\ppp\mid\www}\left(1-\frac{1}{\norm\ppp}\right)\bomega_{\ppp}(s,t) \gg_\www 1,
  \end{equation*}
  for all $s,t\in\classrep$ with $(s,t)\equiv (s_0,t_0)\bmod \www^l$. The implicit constant is independent of $s,t$.
\end{corollary}

\begin{proof}
  Let $\ppp\mid\www$. Lemma \ref{lem:trivial_fibre_bundle}, when combined with Lemma \ref{lem:local density continuous}, yields an open ball $\mathcal{U}_0=\{(s,t)\in K_\ppp^2\ :\  (s,t)\equiv (s_0,t_0)\bmod \ppp^l\}$, on which $\bomega_v$ is continuous and such that
  \begin{equation}\label{eq:Kp points exist}
C_{(s,t)}(K_\ppp)\neq\emptyset,
\end{equation}
and $\Delta(s,t)\neq 0$ for all $(s,t)\in \mathcal{U}_0$.  Since $\mathcal{U}_0$ is also compact, the density $\bomega_\ppp(s,t)$ has a minimal value on $\mathcal{U}_0$, which is non-zero due to \eqref{eq:Kp points exist}. To prove the corollary, we choose $l$ maximal for all $\ppp\mid\www$.
\end{proof}

Our next application is to the archimedean places. Recall the definition of $\c G$ in \eqref{eq:def_fund_dom}.

\begin{proposition}\label{prop:local density arch bound}
  There are compact balls $\mathcal{D}_v \subset K_v^2\smallsetminus\{0\}$, for $v\in\archplaces$, such that their product $\mathcal{D} := \prod_{v\in\archplaces}\mathcal{D}_v\subset K_\infty^2$ has the following properties:
  \begin{enumerate}
  \item $(s_0,t_0)$ is in the interior of $\mathcal{D}$. \label{item:1}
  \item $\mathcal{D} \subset \mathcal{G}$.
  \item For all $v\in \archplaces$ and all $(s_v,t_v)\in T\mathcal{D}_v$, $T>0$, we have $\Delta(s_v,t_v)\neq 0$.
  \item For every $p\in M$, every real place $v$, every linear factor $L(s,t)$ of $\Delta_p(s,t)\in K_v[s,t]$, every $T>0$, and every $(s_v,t_v)\in T\mathcal{D}_v$, the sign of $L(s_v,t_v)$ equals the sign of $L(s_0,t_0)$ in $K_v=\RR$.  
  \item For all $(s,t)\in K^2\cap T\mathcal{D}$, $T>0$, the conic $C_{(s,t)}$ is smooth and has $K_v$-points for all $v\in\archplaces$.
  \item For $(s,t)\in K^2 \cap T\mathcal{D}$, $T>0$, we have
    \begin{equation*}
      \prod_{v\in\archplaces}\bomega_v(s,t) \gg_{\mathcal{D}} T^{-2\dg},
    \end{equation*}
        where the implicit constant is independent of $s,t$.
  \end{enumerate}
\end{proposition}

\begin{proof}
  For $v\in\archplaces$, let $\mathcal{D}_v$ be a compact ball contained in the set $\mathcal{U}_0$ from Lemma \ref{lem:trivial_fibre_bundle}, such that $(s_0,t_0)$ is in the interior of $\mathcal{D}_v$. Then  $(1)$ is clear, and since $(s_0,t_0)$ is in the interior of $\mathcal{G}$, we can achieve $(2)$ by shrinking the $\mathcal{D}_v$. Moreover, since $\Delta(s_0,t_0)\neq 0$, we may achieve $(3)$ and $(4)$ by further shrinking the balls $\mathcal{D}_v$. Property $(5)$ is clear from Lemma \ref{lem:trivial_fibre_bundle}, so it remains to prove $(6)$. To this end, we use the homogeneity condition \eqref{eq:homogeneity condition} on $Q(s,t,x_0,x_1,x_2)$ and the shape \eqref{eq:local height shape} of the local height factors $H_{(s,t),v}=H_{v}\circ\phi_{(s,t)}$ at archimedean places.

Fix an archimedean place $v$, and let $(s,t)\in\mathcal{D}_v$. Since $C_{(s,t)}$ is a smooth conic, the subset of $(x_0:x_1:x_2)\in C_{(s,t)}(K_v)$ with $x_0=0$ or $Q_{(s,t),x_1}(x_0,x_1,x_2)=0$ has dimension $0$ and thus measure $0$. Denote its complement by $C_{(s,t)}'$. 
Analogously, the subset of $C_{T(s,t)}(K_v)$ described by the analogous equations $x_0=0$ or $Q_{T(s,t),x_1}(x_0,x_1,x_2)=0$ has measure zero, and we denote its complement by $C_{T(s,t)}'$.

Let $\theta : C_{T(s,t)}' \to C_{(s,t)}'$ be the diffeomorphism that maps $(1:x_1:x_2) \mapsto (1:T^{a_1}x_1:T^{a_2}x_2)$. 
The implicit function theorem provides an atlas $(\varphi_i:V_i\to W_i)_{i\in I}$ for $C_{(s,t)}'$ consisting of charts of the form $\varphi_i(1:x_1:x_2) = x_2$, mapping an open subset $V_i \subseteq C_{(s,t)}'$ diffeomorphically to an open subset of $W_i\subseteq K_v$. Let $(f_i)_{i\in I}$ be a partition of unity for $C_{(s,t)}'$ that satisfies $\supp f_i \subseteq V_i$.

We construct now an atlas for $C_{T(s,t)}'$ and a corresponding partition of unity as follows. The open subsets $V_{T,i} := \theta^{-1}(V_i)\subseteq C_{T(s,t)}'$, for $i\in I$ cover $C_{T(s,t)'}$. Let $W_{T,i} := T^{-a_2}W_i\subseteq K_v^2$, then 
\begin{equation*}
  \varphi_{T,i} := T^{-a_2}\theta^*\varphi_i : V_{T,i}\to W_{T,i},\quad (1:x_1:x_2) \mapsto T^{-a_2}\varphi_i(\theta(1:x_1:x_2)) = x_2,
\end{equation*}
is a chart for $C_{T(s,t)}'$. The family $(\varphi_{T,i})_{i\in I}$ is an atlas for $C_{T(s,t)}'$, and we have a corresponding partition of unity $(\theta^*f_i)_{i\in I}$, where $\theta^*f_i := f_i\circ\theta$. Hence, we obtain the following formulas for the $v$-adic densities:
\begin{align*}
  \bomega_v(s,t) &= \int_{C_{(s,t)}'}\bomega_{H_{(s,t),v}} = \sum_{i\in I}\int_{W_i}\frac{f_i(\varphi_i^{-1}(x_2))\mathrm{d}x_2}{H_{(s,t),v}(\varphi_i^{-1}(x_2))^\locdegv\absv{Q_{(s,t),x_1}(\varphi_i^{-1}(x_2))}^\locdegv},\\
  \bomega_v(T(s,t)) &= \int_{C_{T(s,t)}'}\bomega_{H_{T(s,t),v}} = \sum_{i\in I}\int_{W_{T,i}}\frac{\theta^*f_i(\varphi_{T,i}^{-1}(x_2))\mathrm{d}x_2}{H_{T(s,t),v}(\varphi_{T,i}^{-1}(x_2))^\locdegv\absv{Q_{T(s,t),x_1}(\varphi_{T,i}^{-1}(x_2))}^\locdegv}.
\end{align*}
Omitting the index $i$ for clarity, let $\varphi:V\to W$ be one of the charts for $C_{(s,t)}'$, $f$ the corresponding summand in the partition of unity, and $\varphi_T:\theta^{-1}(V)\to W_T$ the corresponding chart for $C_{T(s,t)}'$. Replacing the variable $x_2\in W_{T}$ by $T^{a_2}x_2\in W$, we see that
\begin{align*}
  &\int_{W_{T}}\frac{\theta^*f(\varphi_{T}^{-1}(x_2))\mathrm{d}x_2}{H_{T(s,t),v}(\varphi_{T}^{-1}(x_2))^\locdegv\absv{Q_{T(s,t),x_1}(\varphi_{T}^{-1}(x_2))}^\locdegv}\\ = &\int_{W}\frac{\theta^*f((\theta^*\varphi)^{-1}(x_2))T^{-a_2\locdegv}\mathrm{d}x_2}{H_{T(s,t),v}((\theta^*\varphi)^{-1}(x_2))^\locdegv\absv{Q_{T(s,t),x_1}((\theta^*\varphi)^{-1}(x_2))}^\locdegv}. 
\end{align*}
With the observation that
\begin{align*}
  H_{T(s,t),v}(1,x_1,x_2) &= T^{2-a_1-a_2-e}(H_{(s,t),v}\circ\theta)(1,x_1,x_2),\\
  Q_{T(s,t),x_1}(1,x_1,x_2) &= T^{e+a_1}(Q_{(s,t),x_1}\circ\theta)(1,x_1,x_2),
\end{align*}
 by \eqref{eq:local height shape} and \eqref{eq:homogeneity condition},
 this last integral is equal to 
 \begin{equation*}
   \frac{1}{T^{2\locdegv}}\int_{W}\frac{f(\varphi^{-1}(x_2))\mathrm{d}x_2}{H_{(s,t),v}(\varphi^{-1}(x_2))^\locdegv\absv{Q_{(s,t),x_1}(\varphi^{-1}(x_2))}^\locdegv}.
 \end{equation*}
We have shown that $\bomega_v(T(s,t)) = T^{-2\locdegv}\bomega_v(s,t)$. Since $\bomega_v$ is continuous on $\mathcal{U}_0$ by Lemma \ref{lem:local density continuous}, it has a minimum in $\mathcal{D}_v$, which is non-zero thanks to $(5)$. This shows that $\bomega_v(s,t)\gg_{\mathcal{D}_v} T^{-2\locdegv}$ for $(s,t)\in T\mathcal{D}_v$, which proves $(6)$.
\end{proof}

Extending the archimedean factor of the Weil height in the obvious way from $K^2$ to $K_\infty^2$, we see that $H_\infty(Ts,Tt) = T^{\dg}H_\infty(s,t)$ holds for any $T>0$. Choose $\theta>1$ such that $\mathcal{D}\cap\theta\mathcal{D} = \emptyset$.  Moreover, let $N$ be the largest
integer such that all $(s,t)\in \theta^N\mathcal{D}$ satisfy $H(s,t)\leq \norm\classrep B^\delta$, in particular
$N \asymp_{\mathcal{D},\theta,\delta} \log B$.

Fix any ideal $\www \in \ideals$ divisible by $2\classrep$, and let $l$ be as in Corollary \ref{cor:local density nonarch bound}. Since the set
\begin{equation*}
 \{(s,t)\in\mathcal{G}\where H(s,t)_\infty\leq \norm\classrep B^\delta\}
\end{equation*}
contains the disjoint union $\mathcal{D}\cup \theta\mathcal{D}\cup\dots\cup\theta^N\mathcal{D}$, we see that
\begin{equation*}
  \mathfrak{S}(B) \gg_{\www,\mathcal{D},\theta}\sum_{k=0}^N\theta^{-2k\dg}\sum_{\substack{(s,t)\in \classrep^2\cap\theta^k\mathcal{D}\\s\OO_K + t\OO_K = \classrep\\(s,t)\equiv (s_0,t_0)\bmod \www^l}}\prod_{\ppp\nmid\www}\left(1-\frac{1}{\norm\ppp}\right)\bomega_\ppp(s,t). 
\end{equation*}
To obtain our desired bound \eqref{eq:desired bound 1}, it is thus enough to show that
\begin{equation}\label{eq:desired bound 2}
  \sum_{\substack{(s,t)\in \classrep^2\cap B\mathcal{D}\\s\OO_K + t\OO_K = \classrep\\(s,t)\equiv (s_0,t_0)\bmod \www^l}}\prod_{\ppp\nmid\www}\left(1-\frac{1}{\norm\ppp}\right)\bomega_\ppp(s,t)\gg_{\mathcal{D},\www} B^{2\dg}(\log B)^{\card\{p\in M\where X_p \text{ split}\}}, 
\end{equation}
for $B>1$. We have the freedom to multiply $\www$ by ``bad'' prime ideals, which we will do several times during the course of the proof. We start by requiring that $\www$ is divisible by all prime ideals in the set $S_{\text{bad}}$ from \S\ref{sec:conic bundles and divisors}

\subsection{Detectors}
Now we apply the results of the previous subsection and \S\ref{sec:local_densites_conics} to the analysis of the densities
\begin{equation*}
  \sigma_\ppp(s,t) := \left(1-\frac{1}{\norm\ppp}\right)\bomega_\ppp(s,t)
\end{equation*}
in \eqref{eq:desired bound 2}. Recall that $M$ is the set of all closed points $p$ in $\P^1_K$ such that the fibre $X_p$ is singular, and that we have a factorisation \eqref{eq:disc_fact}.
From now on, we assume that $\www$ is included in all prime ideals dividing $a$ or any coefficient of any of the forms $\Delta_p(s,t)$. Recall moreover the definitions of $\theta_p$ and $\delta_p$ from \S\ref{sec:conic bundles and divisors}.

\begin{lemma}\label{lem:ref_T_diagonalization}
  Let $p \in M$. Then $\delta_p(\theta_p,1) \neq 0$. Moreover, the fibre $X_p$ is split over $K(p)$ if and only if $\delta_p(\theta_p,1)$ is a square in $K(p)$.
\end{lemma}

\begin{proof}
Recall that $X_p$ is isomorphic to the plane conic $C_{(\theta_p,1)}$. We assume without loss of generality that $i_p = 0$. Let $\vv_p=(v_0,v_1,v_2) \in \OO_{K(p)}^3$ be a representative of the unique singular point of $C_{(\theta_p,1)}$, and define the invertible (over $K(p)$) matrix
  \begin{equation*}
    T :=
    \begin{pmatrix}
      v_0 & 0 & 0\\
      v_1 & v_0 & 0\\
      v_2 & 0 & v_0
    \end{pmatrix}.
  \end{equation*}
Then, with $\vy := T^{-1}\vx$, 
  \begin{align*}
    Q_{(\theta_p,1)}(\vx) &= Q_{(\theta_p,1)}(v_0(0,y_1,y_2) + y_0\vv_p)\\
             &= v_0^2Q_{(\theta_p,1)}(0,y_1,y_2) + y_0^2Q_{(\theta_p,1)}(\vv_p) + y_0v_0\nabla Q_{(\theta_p,1)}(\vv_p)\cdot (0,y_1,y_2)\\
             &= v_0^2Q_{(\theta_p,1)}(0,y_1,y_2). 
  \end{align*}
Hence $C_{(\theta_p,1)}(\vx)$ is split over $K(p)$ if and only if the binary quadratic form $Q_{(\theta_p,1)}(0,y_1,y_2)$ is, which holds if and only if its discriminant is a square. Moreover, $\delta_p(\theta_p,1) \neq 0$ since otherwise $C_{(\theta_p,1)}$ would be a double line, which does not occur by Lemma \ref{lem:Pic}.
\end{proof}

\begin{lemma}
  The resultant $\res(\Delta_p(s,t),\delta_p(s,t))$ of the binary forms $\Delta_p(s,t)$ and $\delta_p(s,t)$ is not zero.
\end{lemma}

\begin{proof}
  If $\res(\Delta_p(s,t),\delta_p(s,t))=0$ then $\Delta_p(s,t)$ and $\delta_p(s,t)$ would have a non-constant common factor. Since $\Delta_p(s,t)$ is irreducible over $K$, this would imply that $\Delta_p(s,t) | \delta_p(s,t)$, and thus $\delta_p(\theta_p,1)=0$, contradicting Lemma \ref{lem:ref_T_diagonalization}.
\end{proof}

In the proof of the following lemma, we will modify $\www$ several times to ensure that its factorisation includes certain small primes. Recall that we already assume that $2\classrep\mid\www$. 

We need to introduce some notation to deal with the case when the polynomial $\Delta_{p}(x,1)$, for $p\in M$, is not monic. Let $b_p := \Delta_p(1,0)$ be the leading coefficient, and define the monic irreducible polynomial $\widetilde{\Delta}_p(x) := b_p^{\deg\Delta_p-1}\Delta_p(b_p^{-1}x,1)\in\OO_K[x]$. Moreover, let $\widetilde{\theta}_p := b_p\theta_p$. Then $\widetilde\theta_p$ is a root of $\widetilde\Delta_p$, and thus an algebraic integer. Moreover, $K(p) = K(\widetilde\theta_p)$. 

Recall, that for nonzero ideals $\aaa$ of $\OO_K$ and $\mathfrak{B}$ of $\OO_{K(p)}$, we write
\begin{equation*}
  \aaa^\flat = \prod_{\ppp\nmid\www}\ppp^{v_\ppp(\aaa)},\quad \mathfrak{B}^\flat = \prod_{\PPP\nmid\www\OO_{K(p)}}\PPP^{v_\PPP(\mathfrak{B})},
\end{equation*}
for the prime-to-$\www$ parts of $\aaa$, respectively $\mathfrak{B}$.

\begin{lemma}\label{lem:detector}
  Let $\ppp\nmid\www$ be a prime ideal of $\OO_K$ and $s,t\in\OO_K$ with $s\OO_K + t\OO_K=\classrep$ such that $\Delta(s,t)\neq 0$.
  \begin{enumerate}
  \item If $\ppp\mid\Delta(s,t)$, then there is a unique $p\in M$ with $\ppp\mid\Delta_p(s,t)$. Moreover, $\ppp\nmid\delta_p(s,t)$.
  \end{enumerate}
  Now assume that $\ppp\mid \Delta_p(s,t)$ for some $p\in M$, and $\ppp\nmid\www$.
  \begin{enumerate}
  \setcounter{enumi}{1}
  \item There is a unique prime ideal $\PPP$ of $\OO_{K(p)}$ above $\ppp$ that divides $b_ps-\widetilde\theta_p t$.
  \item The prime ideal $\PPP$ satisfies $\OO_{K(p)}/\PPP\cong\OO_K/\ppp$ and $v_\ppp(\Delta_p(s,t)) = v_\PPP(b_p s - \widetilde\theta_p t)$. 
  \item The element $\delta_p(\theta_p,1)\in K(p)$ is defined modulo $\PPP$, and $\qrp{\delta_p(s,t)} = \qr{\delta_p(\theta_p,1)}{\PPP}$. 
  \item The rank of $Q_{(s,t)}$ over $\OO_K/\ppp$ is $2$, and $\chi_\ppp(Q_{(s,t)}) = \qrp{\delta_p(s,t)}$. In particular, we have $\chi_\ppp(Q_{(s,t)})=1$ when $X_p$ is split.
 \item The identity $\qr{\delta_p(s,t)}{\Delta_p(s,t)^\flat} = \qr{\delta_p(\theta_p,1)}{(b_ps-\widetilde\theta_p t)^\flat}$ holds between Jacobi symbols over $K$ and $K(p)$.
  \end{enumerate}
\end{lemma}

\begin{proof}
Let us modify $\www$ to ensure it is divisible by all prime ideals dividing
  \begin{equation*}
a\prod_{p_1\neq p_2\in M}\res(\Delta_{p_1},\Delta_{p_2})\prod_{p\in M}\res(\Delta_p,\delta_p)\prod_{p\in M}b_p.
\end{equation*}
Then (1) is clear, and moreover the condition $s\OO_K+t\OO_K = \classrep\mid\www$ implies $\ppp\nmid t$. Hence, $st^{-1}\bmod \ppp$ exists and is a root of the polynomial $\Delta_p(x,1)$ modulo $\ppp$. Then $b_pst^{-1}$ is a root of $\widetilde{\Delta}_p(x)$ modulo $\ppp$. A further modification of $\www$ allows us to assume that $\ppp\OO_{K(\theta_p)}$ is coprime to the conductor of the order $\OO_K[\widetilde\theta_p]$ in $K(\theta_p)$. Then every prime $\PPP$ of $\OO_{K(p)}$ above $\ppp$ is uniquely determined by $\PPP\cap \OO_K[\widetilde\theta_p]$, and 
$\OO_K[\widetilde\theta_p]/(\PPP\cap \OO_K[\widetilde\theta_p])\cong \OO_{K(p)}/\PPP$.
Since 
\begin{equation*}
\ppp \mid b_p^{\deg\Delta_p-1}\Delta_p(s,t) = b_p^{\deg\Delta_p}\prod_{\substack{\sigma: K(p) \to \overline K\\K\text{-embedding}}}(s-\theta_p^\sigma t) = N_{K(p)|K}(b_ps-\widetilde\theta_p t),
\end{equation*}
there are prime ideals above $\ppp$ that divide $b_ps-\widetilde\theta_p t$. Let $\PPP$ be any such prime ideal. Then $b_pst^{-1}\equiv \widetilde\theta_p \bmod\PPP$, so the natural embedding $\OO_K/\ppp \hookrightarrow \OO_K[\widetilde\theta_p]/(\PPP\cap\OO_K[\widetilde\theta_p])$ is an isomorphism. The inverse is induced by the map $\OO_K[\widetilde\theta_p]\to \OO_K/\ppp$, $\widetilde\theta_p\mapsto b_pst^{-1}\bmod\ppp$, so $\PPP\cap\OO_K[\widetilde\theta_p]$ is uniquely determined as the kernel of this map. This shows $(2)$.

For $(3)$, we have already shown that $\OO_K/\ppp\cong\OO_K[\widetilde\theta_p]/(\PPP\cap\widetilde\theta_p)\cong\OO_{K(\theta_p)}/\PPP$. Splitting off roots, we see that
\begin{equation*}
  b_p^{\deg\Delta_p-1}\Delta_p(s,t) = (b_p s - \widetilde\theta_p t)R_p(s,t)\text{ and }\widetilde\Delta_p(x) = (x-\widetilde\theta_p)R_p(b_p^{-1}x,1),
\end{equation*}
for some binary form $R_p$ over $\OO_{K(p)}$ of degree $\deg\Delta_p-1$. Suppose that $v_\PPP(R_p(s,t))>0$. Then 
\begin{equation*}
  R_p(st^{-1},1) = t^{-\deg\Delta_p+1}R_p(s,t)=0,  
\end{equation*}
so $b_pst^{-1}$ is a multiple root of $\widetilde\Delta_p(x)$ in $\OO_{K(p)}/\PPP \cong \OO_K/\ppp$. Hence, $\ppp$ is ramified in $K(p)$, which we can exclude by modifying $\www$. Therefore, we may assume that $v_\PPP(R_p(s,t))=0$, and thus
\begin{equation*}
  v_\ppp(\Delta_p(s,t)) = v_\PPP(\Delta_p(s,t)) = v_\PPP(b_p s - \widetilde\theta_p t).
\end{equation*}

For $(4)$, we invert $b_p$ to obtain $\theta_p\equiv st^{-1}\bmod\PPP$. Thus, $\delta_p(\theta_p,1)$ exists modulo $\PPP$, and
\begin{equation*}
  \qr{\delta_p(\theta_p,1)}{\PPP}=\qr{\delta_p(st^{-1},1)}{\PPP}=\qr{\delta_p(st^{-1},1)}{\ppp} = \qrp{\delta_p(s,t)}.
\end{equation*}
In the last equality, we used the fact that $\delta_p$ has even degree. 

Let us prove $(5)$. The assertion about the rank follows, since $\ppp\nmid\delta_\ppp(s,t)$. The conics $C_{(s,t)}$ and $C_{(st^{-1},1)}$ are isomorphic modulo $\ppp$ due to \eqref{eq:homogeneity condition} and since $t$ is invertible modulo $\ppp$. Hence,
\begin{equation*}
  \chi_\ppp(Q_{(s,t)}) =\chi_\ppp(Q_{(st^{-1},1)}) = \chi_\PPP(Q_{(\theta_p,1)}).
\end{equation*}
In the proof of Lemma \ref{lem:ref_T_diagonalization}, we have seen a transformation $x = Ty$ giving (w.l.o.g with $i_p=0$)
\begin{equation*}
  Q_{(\theta_p,1)}(\vx) = v_0^2Q_{(\theta_p,1)}(0,y_1,y_2).
\end{equation*}
By including $\www$ in all prime ideals dividing $v_{i_p}$, we ensure that this transformation is valid modulo $\PPP$, and thus $\chi_\PPP(Q_{(\theta_p,1)}) = \chi_\PPP(Q_{(\theta_p,1)}(0,y_1,y_2))$. Since $\chi_\PPP(Q) \in \{ \pm 1 \}$ depending on whether $Q$ is split or non-split modulo $\PPP$, we see that 
\begin{equation*}
  \chi_{\PPP}(Q_{(\theta_p,1)}(0,y_1,y_2)) = \qr{\delta_p(\theta_p,1)}{\PPP}= \qrp{\delta_p(s,t)}.
\end{equation*}

To prove $(6)$, we see that by $(2)$, $(3)$ and $(4)$,
\begin{equation*}
  \qr{\delta_p(s,t)}{\Delta_p(s,t)^\flat} = \prod_{\substack{\ppp\mid\Delta_p(s,t)\\\ppp\nmid\www}}\qrp{\delta_p(s,t)}^{v_\ppp(\Delta_p(s,t))}=\prod_{\substack{\PPP\mid (b_p s-\widetilde\theta_p t)\\\PPP\nmid\www}}\qr{\delta_p(\theta_p,1)}{\PPP}^{v_\PPP(b_p s -\widetilde\theta_p t)}.  \qedhere
\end{equation*}
\end{proof}

\subsection{Estimates for the local densities}
\begin{lemma}\label{lem:local_densities_estimate}
  Let $s,t \in \OO_K$ with $s\OO_K + t\OO_K = \classrep$ and such that $\Delta(s,t)\neq 0$. Let $\ppp$ be a prime ideal of $\OO_K$ with $\ppp\nmid \www$ and $\ppp \mid \Delta(s,t)$. Then there is a unique $p\in M$ such that $\ppp\mid\Delta_p(s,t)$, and moreover
\begin{equation*}
  \sigma_\ppp(s,t)\geq \left(1 -\frac{2}{\norm\ppp}\right)\sum_{k=0}^{v_\ppp(\Delta_p(s,t))}\qrp{\delta_p(s,t)}^k.
\end{equation*}
If $X_p$ is split, this implies in particular that
\begin{equation*}
  \sigma_\ppp(s,t) \geq \left(1-\frac{2}{\norm\ppp}\right)\left(1+v_\ppp(\Delta_p(s,t))\right).
\end{equation*}
\end{lemma}

\begin{proof}
  Take $\ppp$ as in Lemma \ref{lem:detector}. Then all hypotheses of Proposition \ref{prop:local density bounds} are satisfied, and the first inequality follows. The second one is an immediate consequence. 
\end{proof}

Recall the definition of $r(s,t)$ in \S\ref{sec:conic bundles and divisors}.

\begin{lemma}
  For $s,t\in\OO_K$ with $s\OO_K+t\OO_K = \classrep$ and such that $\Delta(s,t)\neq 0$, we have
  \begin{equation*}
    \prod_{\ppp\nmid\www}\sigma_\ppp(s,t) \gg r(s,t).
  \end{equation*}
\end{lemma}

\begin{proof}
  For any prime ideal $\ppp\nmid\Delta(s,t)\www$, we have
  \begin{equation*}
    \sigma_\ppp(s,t)=1-\frac{1}{\norm\ppp^2},
  \end{equation*}
  according to Lemma \ref{lem:local_smooth}. Together with the estimations from Lemma \ref{lem:local_densities_estimate}, this shows that the product on the left-hand side in the lemma is
  \begin{equation*}
    \gg \prod_{\substack{p\in M\\X_p \text{ split}}}\prod_{\substack{\ppp\mid\Delta_p(s,t)\\\ppp\nmid \www}}(1+v_\ppp(\Delta_p(s,t)))\left(1-\frac{2}{\norm\ppp}\right)\hspace{-0.4cm}\prod_{\substack{p\in M\\X_p \text{ non-split}}}\prod_{\substack{\ppp\mid\Delta_p(s,t)\\\ppp\nmid \www}}\left(1-\frac{2}{\norm\ppp}\right)\sum_{k=0}^{v_\ppp(\Delta_p(s,t))}\qrp{\delta_p(s,t)}^k.
  \end{equation*}
For $X_p$ split, we have
\begin{equation*}
  \prod_{\substack{\ppp\mid\Delta_p(s,t)\\\ppp\nmid \www}}(1+v_\ppp(\Delta_\ppp(s,t))) = \sum_{\ddd\mid\Delta_p(s,t)^\flat}1 = \sum_{\ddd\mid\Delta_p(s,t)^\flat}\qr{\delta_p(s,t)}{\ddd},
\end{equation*}
  by Lemma \ref{lem:ref_T_diagonalization} and \cite[Lemma 3.2]{divsumpub}.
\end{proof}

By \eqref{eq:desired bound 2}, it is thus our goal to prove the estimate
\begin{equation}\label{eq:desired bound r2d2}
  \sum_{\substack{(s,t)\in \classrep^2\cap X\mathcal{D}\\s\OO_K + t\OO_K = \classrep\\(s,t)\equiv (s_0,t_0)\bmod \www^l}}r(s,t)\gg_{\mathcal{D},\www} X^{2\dg}(\log X)^{\card\{p\in M\where X_p\text{ split}\}}, 
\end{equation}
for $X>1$, where $\mathcal{D}$ is as in Proposition \ref{prop:local density arch bound} and $l$ as in Corollary \ref{cor:local density nonarch bound}. This is exactly \eqref{eq:div sum bound} for $\c P = (\mathcal{D}, (s_0,t_0), \www^l)$, so all that is left to show is that the tuple $\c P$ is admissible. Conditions \eqref{eq:f not zero} and \eqref{eq:D not zero} follow immediately from part $(3)$ of Proposition \ref{prop:local density arch bound}. For condition \eqref{eq:qr is one}, we have the following lemma.

\begin{lemma}\label{lem:quad rep}
  Let $(s,t)\in\classrep^2\cap X\mathcal{D}$, for $X>1$, with $s\OO_K + t\OO_K = \classrep$ and $(s,t)\equiv (s_0,t_0)\bmod\www^l$. Then
  \begin{equation*}
    \qr{\delta_p(s,t)}{\Delta_p(s,t)^\flat} = 1
  \end{equation*}
 holds for all
 $p\in M$.
\end{lemma}

\begin{proof}
We may assume that $l\geq 2$. By part $(6)$ of Lemma \ref{lem:detector}, there are $a_p,b_p\in \OO_{K(p)}$, not both zero, such that
\begin{equation*}
    \qr{\delta_p(s,t)}{\Delta_p(s,t)^\flat} = \qr{\delta_p(\theta_p,1)}{(b_ps-a_pt)^\flat},
\end{equation*}
with the second Jacobi symbol taken over $\OO_{K(p)}$. Here, $b_pS-a_pT$ is a linear factor of $\Delta_p(S,T)$ in $\OO_{K(p)}[S,T]$.

Assume, possibly modifying $\www$, that $2^{u}N_{K(p)|K}(\delta_p(\theta_p,1)(b_ps_0-a_pt_0))\mid \www$, where $u$ is large enough to ensure that $1 + 2^{u}\OO_{K(p),w} \subseteq K(p)_w^2$ for every place $w$ of $K(p)$ above $2$. Such an $u$ exists, since squares are an open subgroup of $K(p)_w^\times$. Write $(s,t) = (s_0,t_0) + (v,w)$, where $v,w\in\www^l \subseteq 2^uN_{K(p)|K}(\delta_p(\theta_p,1)(b_ps_0-a_pt_0))\www$. Then
\begin{equation}\label{eq:enable qr}
  b_ps-a_pt = (b_ps_0-a_pt_0) + (b_pv-a_pw) = (b_ps_0-a_pt_0)(1 + \xi),
\end{equation}
with $\xi \in 2^u\delta_p(\theta_p,1)\www\OO_{K(p)}$, and therefore $(b_ps-a_pt)^\flat = (1+\xi)\OO_{K(p)}$. By quadratic reciprocity, we obtain
\begin{equation*}
  \qr{\delta_p(\theta_p,1)}{(b_p s-a_p t)^\flat}  = \qr{1+\xi}{\delta_p(\theta_p,1)}\prod_{w\mid 2\infty}H_w(1+\xi,\delta_p(\theta_p,1)) = \prod_{w\mid 2\infty}H_w(1+\xi,\delta_p(\theta_p,1)).
\end{equation*}
But $H_w(1+\xi,\delta_p(\theta_p,1))=1$ for all archimedean places $w$ of $K(p)$, due to \eqref{eq:enable qr} and Proposition \ref{prop:local density arch bound}, $(3)$. Moreover, $H_w(1+\xi,\delta_p(\theta_p,1))=1$ for all places $w$ dividing $2$ thanks to the fact that $1+\xi\in K(p)_w^2$ by our choice of $u$.
\end{proof}

This completes the proof of Theorem \ref{thm:conic_bundle_precise}, from which Theorem \ref{thm:conic bundle} follows.
For Theorem~\ref{thm:conic bundle small}, recall from \S\ref{sec:conic bundles} that any conic bundle of complexity at most $3$ with a rational 
point is rational, hence contains a rational point lying on a smooth fibre. Thus Theorem \ref{thm:conic bundle small} 
follows from Theorem \ref{thm:conic_bundle_precise} and  \cite[Thm.~1.1]{divsumpub}.
 \qed

\section{Del Pezzo surfaces} \label{sec:del_Pezzo}

In this section we study del Pezzo surfaces with a conic bundle structure. We give necessary and sufficient
criteria for a del Pezzo surface to admit a conic bundle, in terms of the Galois action on the lines.
We show how one deduces the results in \S \ref{sec:del pezzo} and \S \ref{sec:BT}
from Theorem \ref{thm:conic bundle small}, and also classify equations for del Pezzo surfaces with a conic bundle structure using the 
set-up \S \ref{sec:conic_bundles_geometry}. 

Throughout $k$ is a perfect field, assumed for simplicity to not have characteristic $2$ in \S\ref{sec:dp_equation}.

\subsection{Proof of Theorem \ref{thm:BT}}
Let $K$ be a number field and
$$X: \quad x_0^3 + x_1^3 + x_2^3 + x_3^3 = 0 \quad \subset \PP^3_K$$
the Fermat cubic surface over $K$. Let $\pi: X \to \PP^1$ be the conic bundle induced by the planes
containing the line $x_0+x_1=x_2+x_3=0$. A simple calculation using Lemma \ref{lem:Pic} shows that
$$
c(\pi) = 
\begin{cases}
	0, & \text{if } \QQ(\sqrt{-3}) \subseteq K, \\
	2, & \text{if } \QQ(\sqrt{-3}) \not\subset K.
\end{cases}, \quad
\rho(X) = 
\begin{cases}
	7, & \text{if } \QQ(\sqrt{-3}) \subseteq K, \\
	4, & \text{if } \QQ(\sqrt{-3}) \not\subset K.
\end{cases}
$$
In particular, the hypotheses of Theorem \ref{thm:conic bundle small} are satisfied. Let now $Y$ be as in
Theorem \ref{thm:BT} and, for $a=(a_0: a_1:a_2:a_3) \in \PP^3_K$, let $Y_a$ denote
the corresponding diagonal cubic surface.
An anticanonical height on $Y$ induces an anticanonical height on $Y_a$. Moreover, if every $a_i$
is a non-zero cube then $Y_a$ is isomorphic to the Fermat cubic surface. 
As the collection of such $a$ is Zariski dense,
the result follows from a simple application of Theorem \ref{thm:conic bundle small}. \qed

\subsection{Criterion for existence of a conic bundle}

We first show, as is well-known for $d \geq 3$, that any conic on a del Pezzo surface $X$ of degree $d$ gives rise to a conic bundle.
We define a \emph{conic} on $X$ to be a geometrically connected curve $C \subset X$
of arithmetic genus $0$ with $-K_X \cdot C = 2$.
\begin{lemma} \label{lem:existence_of_conic_bundle}
	Let $X$ be a del Pezzo surface over $k$ with a conic $C \subset X$. Then the linear system $|C|$ induces
	a conic bundle structure $X \to \PP^1$, and every conic bundle on $X$ arises this way (for some $C$).
\end{lemma}
\begin{proof}
	Let $C \subset X$ be a conic.
	Consider the exact sequence of sheaves
	\begin{equation} \label{seq:conic}
		0 \to \OO_X \to \OO_X(C) \to \OO_C(C) \to 0
	\end{equation}
	on $X$. By the adjunction formula \cite[Ex.~V.1.3]{Har77} we have $C^2=0$ and, as is well-known, 
	we also have $h^1(X, \OO_X)=0$ \cite[Lem.~III.3.2.1]{Kol96}. Thus applying cohomology to \eqref{seq:conic},
	we find that $\dim |C| = 1$. As $C^2=0$ the linear system $|C|$ is base-point free, hence
	induces a morphism $\pi: X \to \PP^1$. To complete the proof, it suffices to show that the fibres of $\pi$ 
	are isomorphic to plane conics, for which we may assume that $k$ is algebraically closed.

	Let $F$ be a fibre of $\pi$, which is a 
	connected curve of arithmetic genus $0$. As $-K_X \cdot C = 2$ and $-K_X$ is ample, by Nakai-Moishezon 
	\cite[Thm.~V.1.10]{Har77} we see that $C$ has at most $2$ irreducible components. If $C$ is irreducible,
	then, as $p_a(C) = 0$, it is isomorphic to a smooth conic. Otherwise, we have
	$C = C_1 + C_2$ where $C_i$ are irreducible and satisfy $-K_X \cdot C_i = 1$,
	hence are $2$ lines on $X$ meeting in a point, as required.

	For the second part, let $\pi:X \to \PP^1$ be a conic bundle. Then each fibre $C$ has arithmetic genus $0$
	and satisfies $-K_X \cdot C = 2$ by Proposition \ref{prop:conic_bundle}, as required.
\end{proof}

We next obtain a criterion concerning the existence of a conic bundle on del Pezzo surfaces in terms
of the Galois action on the lines. This is required for the computation used in Remark~\ref{rem:Galois_actions} and the 
proof of Theorem~\ref{thm:rho}.

\begin{proposition} \label{prop:dp_conic_bundle}
	Let $X$ be a del Pezzo surface of degree $d\leq 7$ over $k$.
	If $X$ admits a conic bundle structure, then there exist two (possibly identical)
	non-empty Galois orbits $\mathfrak{L}_1$ and $\mathfrak{L}_2$ of lines 
	on $X_{\bar{k}}$ such that $(\mathfrak{L}_1 + \mathfrak{L}_2)^2 = 0$.
	If $X$ has a $0$-cycle of odd degree (e.g.~$d$ is odd or $X(k) \neq \emptyset$),
	then the converse holds.
\end{proposition}
\begin{proof}
	The first implication is elementary. Indeed, as $d\leq 7$ there is at least one $P \in \PP^1$
	for which $\pi^{-1}(P)$ is singular. If $\pi^{-1}(P)$ is non-split then one takes
	$\mathfrak{L}_1 = \mathfrak{L}_2 = \pi^{-1}(P)$. If $\pi^{-1}(P)$ is split then one takes
	$\mathfrak{L}_1$ and $\mathfrak{L}_2$ to be the irreducible components of $\pi^{-1}(P)$.

	We now show the reverse implication. For $i \in \{1,2\}$ write
	$\mathfrak{L}_i = L_{i,1} + \cdots + L_{i,r_i}$ and
	$n_i = L_{i,1} \cdot (\sum_{j \neq 1} L_{i,j})$.
	For $\{i,j\} = \{1,2\}$ let $m_i = L_{i,1} \cdot \mathfrak{L}_j$.
	Note that 
	\begin{equation} \label{eqn:r_im_i}
		r_1m_1 = r_2m_2.
	\end{equation}
	As $\mathfrak{L}_1$ and $\mathfrak{L}_2$ are Galois orbits, 
	from $(\mathfrak{L}_1 + \mathfrak{L}_2)^2 = 0$ we deduce that
	\begin{equation} \label{eqn:r_i}
		\mathfrak{L}_1^2 + \mathfrak{L}_2^2 + 2\mathfrak{L}_1\cdot\mathfrak{L}_2 
		= r_1(n_1-1) + r_2(n_2-1) + 2r_1m_1 = 0.
	\end{equation}
	There are two cases, corresponding to a non-split or a split fibre.
	First suppose that $\mathfrak{L}_1^2 = 0$.
	Then we find that $n_1=1$. Thus $r_1$ is even
	and $\mathfrak{L}_1 \times \bar{k}$ is a collection of $r_1$ lines meeting in pairs. Moreover,
	the Hodge index theorem \cite[Thm~V.1.9]{Har77} implies  that 
	each such pair is linearly equivalent over $\bar{k}$
	(any non-zero totally isotropic subspace of $(\Pic \bar{X}) \otimes \RR$
	is $1$-dimensional).
	It follows from Lemma \ref{lem:existence_of_conic_bundle} that over $\bar{k}$ we have $\mathfrak{L}_1 \times \bar{k} = (r_1/2)F$
	where $F$ is the fibre of a conic bundle. To determine the structure over $k$,
	consider the Stein factorisation of the morphism $f$ induced by $\mathfrak{L}_1$:
	\[\xymatrix{
	X \ar[rr]^f \ar[dr]^\pi & & \PP^n \\ 
	& C \ar[ur]^g &
	} \] 
	Here $\pi$ has connected fibres and $g$ is finite. 
	As Stein factorisation commutes with flat base-change \cite[Tag 03GX]{stacks-project},
	we see that $\pi \times \bar{k}$ is a conic
	bundle, that $C$ is a smooth curve of genus $0$
	and that the map $g$ has degree $r_1/2$ onto its image. However, as $X$ has a $0$-cycle of 
	odd degree, so does $C$ and hence $C \cong \PP^1$ by Riemann--Roch. Thus $\pi$ is a conic bundle in the sense of
	Definition \ref{def:conic_bundle} (here $\mathfrak{L}_1$ corresponds to a non-split fibre of $\pi$).
	
	For the second case, we may assume that $\mathfrak{L}_1^2\mathfrak{L}_2^2 \neq 0$.
	We deduce from \eqref{eqn:r_i} that $n_1 = n_2 = 0$. We obtain
	$$2r_1m_1 = r_1 + r_2.$$
	Using this to eliminate $r_2$ in \eqref{eqn:r_im_i}, we find that
	$$2m_1m_2 = m_1 + 1.$$
	By symmetry we obtain $m_1 = m_2 = 1$ and thus $r_1=r_2$. Hence each $\mathfrak{L}_i$
	consists of a Galois invariant collection of $r_1$ pairwise skew lines, and each line in $\mathfrak{L}_1$
	meets exactly one line in $\mathfrak{L}_2$.
	The proof now proceeds in a similar manner to the previous case. Namely, applying Stein factorisation
	to the morphism determined by $\mathfrak{L}_1 + \mathfrak{L}_2$ gives rise to
	a morphism $\pi:X \to C$ which is a conic bundle over $\bar{k}$. Since $X$ has a $0$-cycle of odd degree,
	we find that $C \cong \PP^1$, as required (here
	$\mathfrak{L}_1 + \mathfrak{L}_2$ corresponds to a split fibre of $\pi$).
\end{proof}

\subsection{Proof of Theorem \ref{thm:rho}}
Under the assumptions of Theorem \ref{thm:rho} we see that $X$ has a $0$-cycle of odd degree.
We claim that if $\rho(X) \geq \rho_d$,
then $X$ admits a conic bundle of complexity at most $3$. It seems difficult to prove this directly via conceptual geometric
arguments. However, the claim is easily verified upon enumerating all possible Galois actions in \texttt{Magma} and using the criterion in Proposition \ref{prop:dp_conic_bundle}. To apply Theorem \ref{thm:conic bundle small}, it therefore
suffices to show that $X(K) \neq \emptyset$. If $d$ is even, then this holds by hypothesis. For $d=5$ and $d=1$ there
is always a rational point: for $d=5$ this is a classical theorem of Enriques \cite[Cor.~3.1.5]{Sko01},
and for $d=1$ a rational point is given by the unique base-point of the linear system $|-K_X|$. For $d=3$,
a cubic surface with a conic bundle has a line, hence a rational point. \qed

\subsection{Proof of Theorem \ref{thm:degree}}
We now use Theorem \ref{thm:conic bundle small} to prove Theorem \ref{thm:degree}. To do so, it suffices
to show the following.

\begin{lemma} \label{lem:n_d}
	Let $d \leq 5$, let $n_d$ be as in Table \ref{tab:n_d} and
	let $X$ be a del Pezzo surface of degree $d$ over $k$. Then
	there exists a field extension $L/k$ of degree at most $n_d$ 
	such that $X(L) \neq \emptyset$ and such that $X_L$ admits a conic bundle of complexity at most $3$. 
\end{lemma}

\begin{proof}
	We prove the result using Galois theory and facts about the configuration of
	lines on del Pezzo surfaces.
	
	Recall the definition of the graph $G_d$ from Remark \ref{rem:Galois_actions}.
	Given a collection of vertices $V$ of $G_d$, we need to understand which subgroup of $W(\E_{9-d})$ leaves
	$V$ invariant. Due to the geometric nature of $G_d$, we can do this by constructing del Pezzo surfaces
	whose splitting fields have Galois groups which 
	leave the relevant configurations of lines invariant. We explain this method in 
	detail in the cases of interest.

	\smallskip 
	$d=5:$  Let $P$ be a closed point of degree $4$ in $\PP^2_\QQ$ in general position. 
	The blow-up of $\PP_\QQ^2$ at $P$ is a del Pezzo surface $X$ of degree $5$ over $\QQ$. This admits a conic bundle
	structure; explicitly, this structure arises from 
	the pencil of conics in $\PP^2$ which contain $P$. If $P$ is sufficiently general, then the splitting field
	of $X$ has Galois group $S_4$ and the associated conic bundle has complexity $3$.

	We conclude the following about $G_5$: Let $V$ be a collection 
	of $4$ pairwise skew vertices in $G_5$. The subgroup of $W(\E_4)$ which leaves $V$ invariant is a subgroup
	of $S_4$; but by the above we see that it also contains $S_4$, hence is $S_4$. By Galois theory we find
	that, given a del Pezzo surface $X$ of degree $5$ over $k$,
	 there is a field extension $L/k$ of degree at most 
	$(\# W(\E_4) / \#S_4) = 5 = n_5$ such that $X_L$ contains a Galois invariant collection of $4$ pairwise 
	skew lines over $\bar{L}$. As $X(k) \neq \emptyset$ \cite[Cor.~3.1.5]{Sko01}, we find
	that $X_L$ is the blow-up of $\PP_L^2$ in a collection of closed points of total degree $4$,
	hence conclude, as above, that $X_L$ has a conic bundle of complexity at most $3$.

	\smallskip 
	$d\leq 4:$  We use a variant of the previous construction. Let $P \in \PP^2_{\QQ}(\QQ)$ and
	let $Q$ be a closed point of $\PP^2_\QQ$ of degree $(8-d)$ such that $P \sqcup Q$ lies in general position.
	The blow-up of $\PP_\QQ^2$ at $P \sqcup Q$ is a del Pezzo surface of degree $d$. It has a conic bundle
	structure of complexity $0$, arising from the pencil of lines through $P$.
	If $Q$ is chosen sufficiently generally, the  splitting field of $X$ has Galois group $S_{8-d}$.
	
	As in the case $d=5$, by Galois theory and the properties of $G_d$
	we conclude that for a del Pezzo surface $X$ of degree $d$ over $k$,
	there exists a field extension $L/k$ of degree at most $(\# W(\E_{9-d}) / \#S_{8-d})$ such that
	$X_L$ is isomorphic to the blow-up of $\PP_L^2$ in a rational point and collection of closed points
	of total degree $(8-d)$. This admits a conic bundle of complexity $0$ and clearly has a rational point.
	The lemma is proved on noting that $n_d = \# W(\E_{9-d}) / (8-d)!$ for $d\leq 4$.
\end{proof}

This completes the proof of Theorem \ref{thm:degree}. \qed

\begin{remark}
	The astute reader will notice that in the proof of Lemma \ref{lem:n_d},
	for $d\leq 4$ we use del Pezzo surfaces with a conic bundle
	of complexity $0$, however Theorem \ref{thm:conic bundle small} applies whenever there is a conic
	bundle of complexity at most $3$. It turns out that the bounds obtained in 
	Theorem \ref{thm:degree} are the best one may obtain using Theorem \ref{thm:conic bundle small}.
	
	Namely, let $X$ be a del Pezzo surface of degree at most $4$ which admits a conic bundle
	of complexity at most $3$. Then enumerating all possible Galois action in \texttt{Magma}
	and using the criterion from Proposition \ref{prop:dp_conic_bundle}, one finds
	that the splitting field of $X$ always has degree at most $(8-d)!$.
	Though naturally Theorem \ref{thm:conic bundle small} applies to a much wider range of Galois actions
	than just those with a conic bundle of complexity $0$, as explained in Remark \ref{rem:Galois_actions}. 
\end{remark}



\subsection{Equations for del Pezzo surfaces with a conic bundle structure}

We conclude this section by explicitly describing the equations of del Pezzo surfaces
with a conic bundle structure, with an eye towards assisting future proofs of Manin's conjecture.
We work over a perfect field $k$, assumed to have $\chr(k) \neq 2$
in \S\ref{sec:dp_equation}.

\subsubsection{Cohomological calculations}

We begin with some cohomological calculations.

\begin{lemma} \label{lem:Riemann-Roch}
	Let $X$ be a del Pezzo surface of degree $d$ over $k$ equipped with a conic bundle $\pi: X \to \PP^1$.
	Let $F$ be a fibre of $\pi$ and let $n \in \ZZ$. 
	\begin{enumerate}
		\item[(1)] If $n\geq 0$ then
			$$h^0(X, -K_X + nF) = 3n + d + 1.$$
		\item[(2)] If the linear system $|-K_X + nF|$ has an element which is
			a smooth irreducible curve of genus $0$ then
			$$h^0(X, -K_X + nF) = 4n + d + 2.$$
	\end{enumerate}
\end{lemma}
\begin{proof}
	We prove $(1)$ using induction on $n$. For the case $n=0$, it is well-known that
	\begin{equation} \label{eqn:coh_dP}
		h^0(X, -K_X) = d + 1, \quad h^1(X, -K_X) = h^2(X, -K_X) = 0
	\end{equation}
	(see e.g.~\cite[\S III.3]{Kol96}). For $n>0$, consider
	the following exact sequence of sheaves
	$$0 \to \OO_X(-F) \to \OO_X \to \OO_F \to 0$$
	on $X$. Twisting we obtain
	\begin{equation} \label{eqn:sheaves}
	0 \to \OO_X(-K_X + (n-1)F) \to \OO_X(-K_X + nF) \to \OO_F(-K_X + nF) \to 0.
	\end{equation}
	By Proposition \ref{prop:conic_bundle} we have $(-K_X + nF)\cdot F = 2$, hence 
	\begin{equation} \label{eqn:coh_P1}
	h^0(X,\OO_F(-K_X + nF)) = 3, \quad 
	h^1(X,\OO_F(-K_X + nF)) = 0.
	\end{equation}
	By \eqref{eqn:coh_dP}, for the inductive step we may assume that
	\begin{equation} \label{eqn:h^1=0}
		h^0(X, -K_X + (n-1)F) = 3(n-1) + d + 1, \quad h^1(X,-K_X + (n-1)F) = 0.
	\end{equation}
	 Applying cohomology to \eqref{eqn:sheaves}, and using \eqref{eqn:coh_P1}  and \eqref{eqn:h^1=0}
	 we find that 
	$$
		h^0(X, -K_X + nF) = 3n + d + 1, \quad h^1(X,-K_X + nF) = 0,
	$$
	as required.
	
	For $(2)$, let $C$ be a smooth irreducible curve of genus $0$ in $|-K_X + nF|$.
	Consider the exact sequence
	\begin{equation} \label{eqn:sheaves2}
		0 \to \OO_X \to \OO_X(C) \to \OO_C(C) \to 0
	\end{equation}
	of sheaves on $X$. As $h^1(X, \OO_X) = 0$ \cite[Lem.~III.3.2.1]{Kol96} and $g(C) = 0$, applying cohomology
	to \eqref{eqn:sheaves2} and using Proposition \ref{prop:conic_bundle}, we obtain
	$$h^0(X, -K_X + nF) = 2 + C^2 = 4n + d + 2,$$
	as required.	
\end{proof}

\subsubsection{Equations}
\label{sec:dp_equation}

Assume now that $\chr(k) \neq 2$.
We use the notation
of \S\ref{sec:conic_bundles_geometry} concerning conic bundles. In particular, we denote by $F$
the class of a fibre and $M$ the class of the relative hyperplane bundle,
and write equations for conic bundles in the form
\begin{equation} \label{eqn:conic_bundle2}
	\sum_{0 \leq i,j \leq 2} f_{i,j}(s,t)x_ix_j= 0.
\end{equation}
We follow the convention that $f_{i,j} = f_{j,i}$.

\begin{theorem} \label{thm:dP}
	Let $X$ be a del Pezzo surface of degree $d \leq 5$ over $k$ equipped with a conic bundle $\pi:X \to \PP^1$.
	Then there exists an embedding of $X$ into $\FF(a_0,a_1,a_2)$ as a surface of bidegree $(e,2)$
	which respects $\pi$, with anticanonical class $-K_X$ and,
	when $k$ is a number field, a choice of anticanonical height function $H_{-K_X}$
	given by Table \ref{tab:dP}.
	\begin{table}[ht]
	\centering
	$\begin{array}{|lllll|}
		\hline 
		d & (a_0,a_1,a_2) & (e,2)  & -K_X  &  H_{-K_X}	 \\	\hline
		5 & (0,0,0) & (1,2) & M + F & \prod_v \max\{|s|_v,|t|_v\}^\locdegv \max\{|x_0|_v,|x_1|_v,|x_2|_v\}^\locdegv \\
		4 & (0,1,1) & (0,2) & M & \prod_v\max\{|x_0|_v, |sx_1|_v,|tx_1|_v, |sx_2|_v, |tx_2|\}^\locdegv \\
		3 & (0,0,1) & (1,2) & M & \prod_v\max\{|x_0|_v,|x_1|_v,|sx_2|_v,|tx_2|_v\}^\locdegv \\
		2 & (0,0,0) & (2,2) & M & \prod_v\max\{|x_0|_v,|x_1|_v,|x_2|_v\}^\locdegv \\
		1 & (0,1,1) & (1,2) & M - F & \prod_v\max\{|s^{-1}x_0|_v,|t^{-1}x_0|_v,|x_1|_v,|x_2|_v\}^\locdegv \\
 		\hline
	\end{array}$
	\caption{}
	\label{tab:dP}
	\vspace{-0.6cm}
	\end{table} 
\end{theorem}
\begin{proof}
	Let $d\leq 5$ and let $X$ be a del Pezzo surface of degree $d$ equipped
	with a conic bundle $\pi: X \to \PP^1$. That $X$ admits an embedding into some $\FF(a_0,a_1,a_2)$ follows from
	Lemma \ref{lem:conic_bundle_classification}. Note that 
	by Proposition \ref{prop:conic_bundle}, the degree $d$ and the tuple $(a_0,a_1,a_2)$ determines
	$e$ and hence, applying Proposition \ref{prop:conic_bundle} again, the anticanonical divisor
	and a choice of anticanonical height.
	It therefore suffices to calculate $(a_0,a_1,a_2)$,
	which we do by treating each degree $d$ in turn. 
	
	\smallskip	
	$d=5$: 	Let $\mathfrak{L}$ denote the sum of the lines over $\bar{k}$ which lie in the fibres of $\pi$.
	This consists of $3$ pairs of intersecting lines,
	with each pair being mutually skew. Recall that a del Pezzo surface of degree $5$ has 
	$10$ lines over $\bar{k}$, with intersections determined by the Petersen graph.
	An inspection of the Petersen graph reveals that the $4$ lines not in $\mathfrak{L}$ are
	pairwise skew. As this collection is Galois invariant, we may contract them to obtain a morphism
	$X \to \PP^2$ over $k$. Combining this with $\pi$ gives a map $X \to \PP^1 \times \PP^2$,
	which is easily checked to be a closed immersion. Thus we may take $(a_0,a_1,a_2) = (0,0,0)$, as claimed.
	
	\smallskip
	$d=4$: We explicitly calculate  $(a_0,a_1,a_2)$ using 
	the proof strategy of Lemma \ref{lem:conic_bundle_classification}. Namely, let
	$\omega_\pi^{-1} = \omega_X^{-1} \otimes \pi^*(\omega_{\PP^1})$ be the relative anticanonical bundle of $\pi$.
	As $X$ embeds into $\PP(\pi_* \omega_\pi^{-1})$, it suffices to calculate $\pi_* \omega_\pi^{-1}$.
	Since $\PP(\pi_* \omega_\pi^{-1}) \cong \PP(\pi_* \omega_\pi^{-1} \otimes \OO_X(f))$ for any $f \in \ZZ$, 
	we in fact need only calculate $\pi_* \omega_X^{-1} \otimes \OO_X(f)$ for some $f \in \ZZ$.
	
	Recall that for any sheaf $\mathcal{F}$ on $X$ we have $H^0(X, \mathcal{F}) = H^0(\PP^1, \pi_* \mathcal{F})$,
	and that
	$$
	h^0(\PP^1, \OO(a)) =
	\begin{cases}
		a + 1, &  a \geq 0, \\
		0, &  a < 0.		
	\end{cases}
	$$
	Write $\pi_*(\OO_X(-K_X - F)) = \OO(b_0) \oplus \OO(b_1) \oplus \OO(b_2)$. The divisor $-K_X - F$ is the class
	of a smooth conic. Hence Lemma \ref{lem:Riemann-Roch} implies that
	$$h^0(X, -K_X - F) = 2, \quad h^0(X, -K_X) = 5.$$
	The first equality implies that 
	either $b_1 < 0, b_2 < 0 , b_3 = 1$ or $b_1 < 0, b_2 = 0 , b_3 = 0$. Using the second equality,
	we find that 
	$$(b_0,b_1,b_2) = (-1,-1,1) \mbox{ or } (-1,0,0).$$
	Twisting by $1$ and applying Proposition \ref{prop:conic_bundle} we obtain the possibilities
	$$e=0, \quad(a_0,a_1,a_2) = (0,0,2) \mbox{ or } (0,1,1).$$	
	It suffices to rule out  the case $(a_0,a_1,a_2) = (0,0,2)$.
	Let $X \subset \FF(0,0,2)$ be a surface of bidegree $(0,2)$. 
	As $a_0=a_1=e=0$ and $\chr(k) \neq 2$, one may diagonalise the surface
	so that it has the form
	$$x_0^2 - ax_1^2 = f(s,t)x_2^2,$$
	with $a \in k^*$ and  $\deg f = 4$.
	Such a conic bundle is a \emph{Ch\^{a}telet surface}. It is well known that Ch\^{a}telet surfaces
	are never isomorphic to a del Pezzo surface, as e.g.~they contain $(-2)$-curves over $\bar{k}$
	(see e.g.~\cite[p.~300]{BBP12}). 
	Thus we must have $(a_0,a_1,a_2) = (0,1,1)$, as claimed. 
		
	\smallskip
	$d=3:$ We follow a similar method to the previous case. Write
	$\pi_*(\OO_X(-K_X - F)) = \OO(b_0) \oplus \OO(b_1) \oplus \OO(b_2)$. The divisor 
	$-K_X - F$ is the class of a line on $X$, thus Lemma \ref{lem:Riemann-Roch} yields
	$$h^0(X, -K_X - F) = 1, \quad h^0(X, -K_X) = 4.$$
	From this one obtains $(b_0,b_1,b_2) = (2,2,3)$. Twisting by $-2$ gives the result.

	\smallskip	
	$d=2:$ Combining the conic bundle $\pi:X \to \PP^1$ with the anticanonical map $X \to \PP^2$
	gives a morphism $X \to \PP^1 \times \PP^2$, which is easily checked to be a closed immersion.
	
	\smallskip	
	$d = 1:$ We follow the method for $d = 4$. Write
	$\pi_*(\OO_X(-K_X)) = \OO(b_0) \oplus \OO(b_1) \oplus \OO(b_2)$.
	Lemma \ref{lem:Riemann-Roch} gives
	$$h^0(X, -K_X) = 2, \quad h^0(X, -K_X + F) = 5.$$
	We find that $(b_0,b_1,b_2) = (-1,-1,1)$ or $(-1,0,0)$. Twisting by $1$ gives
	$$e=1, \quad(a_0,a_1,a_2) = (0,0,2) \mbox{ or } (0,1,1).$$	
	We shall rule out the case $(0,0,2)$, using a similar method to
	\cite[p.~398]{lilian_annals}. Let $X$ be a smooth surface of bidegree $(1,2)$ in $\FF(0,0,2)$,
	given by the equation \eqref{eqn:conic_bundle2}.
	Consider the curve $C:x_2=0$ on $X$. We will show that $-K_X \cdot C = -1$,
	which implies that $-K_X$ is not ample by Nakai--Moishezon \cite[Thm.~V.1.10]{Har77}.
	The rational function $t^2x_2/x_0$ on $X$ 
	shows that $C + 2F = M$ in $\Pic X$. Using Proposition \ref{prop:conic_bundle}, we find that
	$M^2=5, C\cdot F = 2$ and $C \cdot M = 1$. We deduce that 
	$-K_X \cdot C = (M - F)\cdot C = -1$, as claimed. Thus $(a_0,a_1,a_2) = (0,1,1)$.
\end{proof}

We now consider the converse of Theorem \ref{thm:dP}. In what follows, for a line bundle $L$ on a variety $X$
a collection of $s_0,\ldots, s_r$ of global sections of $L$, we denote by $|s_0,\ldots, s_r| \subset |L|$
the sub-linear system determined by $s_0,\ldots, s_r$.

\begin{theorem} \label{thm:dP_converse}
	Let $1 \leq d \leq 5$, and let $(a_0,a_1,a_2)$ and $e$ be the 
	corresponding values from Table \ref{tab:dP}. Then a general 
	smooth surface of bidegree $(e,2)$	in $\FF(a_0,a_1,a_2)$ is a del Pezzo surface of degree $d$.
	
	More specifically, let $X$ be a smooth surface of bidegree $(e,2)$
	in $\FF(a_0,a_1,a_2)$ with equation \eqref{eqn:conic_bundle2}.
	 Then we have the following.
	\begin{enumerate}
		\item[] $d=5:$ Every such $X$ is a del Pezzo surface of degree $5$.
		\item[] $d=4:$ $X$ is a quartic del Pezzo surface
		if and only if it is isomorphic to a smooth surface of bidegree $(0,2)$ of the form
		\begin{equation} \label{eqn:dP4}
			x_0^2 + f_{1,1}(s,t) x_1^2 + 2f_{1,2}(s,t) x_1x_2 + f_{2,2}(s,t) x_2^2 = 0 \quad \subset \FF(0,1,1).
		\end{equation}
		\item[] $d=3:$ Write
		\begin{equation} \label{eqn:is_dP3_eqn}
		f_{0,0}(s,t) = a_1s + a_2t, \quad f_{0,1}(s,t)= b_1s + b_2t, \quad f_{1,1}(s,t)= c_1s + c_2t.
		\end{equation}
		Then $X$ is isomorphic to a cubic surface if and only if the scheme
		\begin{equation} \label{eqn:is_dP3}
			a_1x_0^2 + 2b_1 x_0x_1 + c_1x_1^2 =0, \quad  a_2x_0^2 + 2b_2 x_0x_1 + c_2x_1^2 = 0 \quad  \subset \PP^1
		\end{equation}
		is empty.
		\item[] $d=2:$ Write $X$ in the form
		\begin{equation} \label{eqn:dP2}
			X: \, a(x_0,x_1,x_2) s^2 + b(x_0,x_1,x_2) st + c(x_0,x_1,x_2) t^2 = 0 \quad \subset \PP^1 \times \PP^2.
		\end{equation}
		Then $X$ is a  del Pezzo surface of degree $2$
		if and only if the scheme
		\begin{equation} \label{eqn:is_dP2}
		 a(x_0,x_1,x_2) = b(x_0,x_1,x_2) =c(x_0,x_1,x_2) = 0 \quad \subset \PP^2
		\end{equation}
		is empty.
		\item[] $d=1$: 	$X$ is a del Pezzo surface of degree $1$ if and only if
		every element of the linear system $| x_1, x_2 | \subset |-K_X|$ is irreducible.
	\end{enumerate}	
\end{theorem}
\begin{proof}
	Let $X$ be a smooth surface of bidegree $(e,2)$ in 
	$\FF(a_0,a_1,a_2)$ with equation of the shape \eqref{eqn:conic_bundle2}.
	By Proposition \ref{prop:conic_bundle}, 
	the anticanonical divisor $-K_X$ still has the same class as given in Table \ref{tab:dP},
	and moreover $K_X^2 = d$. Thus in each case it suffices to determine whether $-K_X$ is ample.

	\smallskip
	$d = 5:$ The divisor class $M+ F \in \Pic(\PP^1 \times \PP^2)$ is ample. Hence $-K_X$, being the pull-back
	of $M+ F$ to $X$, is ample.	

	\smallskip
	$d = 4:$ First assume that $f_{0,0} \neq 0$ in \eqref{eqn:conic_bundle2}. Completing the square shows that
	$X$ is isomorphic to a surface of the shape \eqref{eqn:dP4}. To see that such a surface
	is a quartic del Pezzo surface, we note that the linear system 
	$|sx_1,tx_1,sx_2,tx_2| \subset |-K_X|$ 
	defines a double cover $X \to \PP^1 \times \PP^1$.
	Thus $-K_X$ is the pull-back of an ample divisor by a finite morphism, hence is ample
	\cite[Ex.~III.5.7]{Har77}. 
	
	Assume now that $f_{0,0} = 0$  in \eqref{eqn:conic_bundle2}. 
	Consider the curve $C: x_1=x_2=0 \subset X$. As $-K_X = M$, we see that $x_0=0 \in |-K_X|$.
	Hence  $-K_X \cdot C = 0$, so $-K_X$ is not ample by Nakai-Moishezon \cite[Thm.~V.1.10]{Har77}
	
	\smallskip	
	$d = 3:$ Consider the map
	\begin{equation} \label{eqn:cubic_f}
		f: X \to \PP^3, \quad f:(s,t;x_0,x_1,x_2) \mapsto (x_0,x_1,sx_2,tx_2).
	\end{equation}
	This is determined by a sub-linear system of $-K_X = M$.
	The image of $f$ is the cubic surface obtained by setting	$x_2=1$ in the equation 
	\eqref{eqn:conic_bundle2}. We claim that $X$ is a del Pezzo surface if and only if
	$f$ is a closed immersion. Indeed, if $f$ is a closed immersion then $\OO_X(-K_X) = f^*\OO_{\PP^3}(1)$
	is ample. Conversely, if $X$ is a del Pezzo surface then $f$ is the anticanonical map,
	which is a closed immersion as $-K_X$ is very ample in this case.
	
	It therefore suffices to show that $f$ is a closed immersion if and only if the
	scheme \eqref{eqn:is_dP3} is empty. 
	The smoothness of $X$ implies that $f$ is an embedding on the open subset $x_2 \neq 0$,
	so it suffices to consider the behaviour along the divisor $x_2 = 0$.
	We find that $f$ is an embedding if and only if for all $(x_0,x_1)$ the scheme
	\begin{equation} \label{eqn:fibre_dP3}
		f_{0,0}(s,t)x_0^2 + 2f_{0,1}(s,t)x_0x_1 + f_{1,1}(s,t)x_1^2 = 0 \quad \subset \PP^1
	\end{equation}
	is a single point. As $\deg f_{0,0} = \deg f_{0,1} = \deg f_{1,1} = 1$, this happens
	if and only if the polynomial in \eqref{eqn:fibre_dP3} does not identically vanish.
	This condition may be viewed as a collection of two quadratic equations in the $(x_0,x_1)$, with coefficients
	given by $(a_i,2b_i,c_i)$ from \eqref{eqn:is_dP3_eqn}. This system has no solution
	if and only if the scheme \eqref{eqn:is_dP3} is empty,
	as required.
			
	\smallskip		
	$d = 2:$ Consider the map 
	$$f: X \to \PP^2, \quad f:(s,t;x_0,x_1,x_2) \mapsto (x_0,x_1,x_2).$$
	This is determined by a sub-linear system of $-K_X = M$.
	If the scheme \eqref{eqn:is_dP2} is empty then $f$
	is a double cover. Hence $-K_X$, being the pull-back of an ample divisor by a finite morphism, is ample.
	Conversely, if $X$ is a del Pezzo surface then $f$ is the anticanonical map, in particular $f$ is
	a finite morphism. One easily sees that \eqref{eqn:is_dP2} must be empty in this case, as required.
	
	\smallskip	
	$d=1$: First assume that $X$ is a del Pezzo surface.  Then 
	it is easy to see that every element of $| x_1, x_2 |=| -K_X |$ is irreducible (e.g.~from the explicit
	model in weighted projective space, one sees that any singular element is either
	a nodal or a cuspidal cubic curve, hence irreducible).
	
	Now assume that every element of $| x_1, x_2 |$ is irreducible. 
	From the equation one sees that $f_{0,0} \neq 0$ and that $| x_1, x_2 |$ has a unique base-point given
	by $f_{0,0}(s,t) = 0$ and $(x_0,x_1,x_2) = (1,0,0)$. In particular, let $C \subset X$
	be an irreducible curve which is not contained in $| x_1, x_2 |$. Then for any $D \in | x_1, x_2 |$
	the intersection $D \cap C$ is finite and non-empty, thus $D \cdot C = -K_X \cdot C > 0$. Next let $C$ be an irreducible
	curve contained in $| x_1, x_2 |$. As every element of $| x_1, x_2 |$ is irreducible
	we find that $C \sim -K_X$, so that $-K_X \cdot C = K_X^2 = 1$.
	Hence $-K_X$ is ample by Nakai-Moishezon \cite[Thm.~V.1.10]{Har77}, as required.
\end{proof}

\begin{remark}
	Let $d,e$, and $(a_0,a_1,a_2)$ be as in Table \ref{tab:dP}.
	Then the criteria given in Theorem~\ref{thm:dP_converse} are non-empty, i.e.~for $d\leq 4$
	there exist smooth surfaces of bidegree $(e,2)$ in $\FF(a_0,a_1,a_2)$ which are \emph{not} del Pezzo
	surfaces of degree $d$.	
	For $2 \leq d \leq 4$ our criteria are very explicit and it is easy to construct such
	examples. Explicit examples for $d=1$ can be found in \cite[p.~398]{lilian_annals}.
	Here it is shown that ``diagonal'' surfaces
	$$f_{0,0}(s,t) x_0^2 + f_{1,1}(s,t) x_1^2 + f_{2,2}(s,t) x_2^2 = 0 \quad \subset \FF(0,1,1)$$
	of bidegree $(1,2)$ are never del Pezzo surfaces of degree $1$.
	
	Note that smooth surfaces of bidegree $(e,2)$ in $\FF(a_0,a_1,a_2)$ are still interesting from the
	perspective of Manin's conjecture, even when they are not del Pezzo surfaces. They are
	often so-called ``generalised'' or ``weak'' del Pezzo surfaces. In particular, Theorem \ref{thm:conic bundle small}
	can also be used to give lower bounds for Manin's conjecture for such surfaces as well.
\end{remark}

\begin{remark}
	Theorems \ref{thm:dP} and \ref{thm:dP_converse}
	imply that any quartic del Pezzo surface over a field of characteristic
	not equal to $2$ with a conic bundle structure can be written in the form \eqref{eqn:dP4}.
	As is well-known,
	such a surface admits a complementary conic bundle structure; with the equation \eqref{eqn:dP4}
	this  is given by mapping onto $(x_1,x_2)$.
	The existence of this second structure played a crucial r\^{o}le in the proof of
	Manin's conjecture for a quartic del Pezzo surface with a conic bundle (see \cite{BB11}).	
	
	Note that $s,t,x_1,x_2$ are a set of generating sections for $-K_X$ in this case.
	In particular, a viable choice for the anticanonical height function is given by
	$$H_{-K_X}(s,t;x) = \prod_v\max\{|s|_v,|t|_v\}^\locdegv\max\{|x_1|_v,|x_2|_v\}^\locdegv.$$
	This choice has the advantage of making it easier to work with the two conic bundle structures.
\end{remark}

\subsection*{Acknowledgements}
We thank Tim Browning, Roger Heath-Brown, J\"{o}rg Jahnel, John Ottem, Matthias Sch\"{u}tt and Tony V\'{a}rilly-Alvarado for helpful conversations.

\bibliographystyle{amsalpha}
\bibliography{fibration}
\end{document}